\documentclass[11pt,reqno]{amsart}
\usepackage{fullpage}

\usepackage{amsmath}
\usepackage{amsfonts}
\usepackage{amsthm}
\usepackage{amssymb}
\usepackage{color}
\usepackage{enumerate}
\usepackage{graphicx}
\usepackage{subfigure}

\usepackage{todonotes}
\newcommand{\Z}{\mathbb{Z}}    
\newcommand{\Q}{\mathbb{Q}} 	
\newcommand{\R}{\mathbb{R}} 	
\newcommand{\bP}{\mathbb{P}}	
\newcommand{\cF}{\mathcal{F}}	
\newcommand{\cB}{\mathcal{B}}	
\newcommand{\bE}{\mathbb{E}}	
\newcommand{\bu}{\overline{u}}	
\newcommand{\cU}{\mathcal{U}}	
\newcommand{\cL}{\mathcal{L}}  
\newcommand{\cS}{\mathcal{S}}	

\newcommand{\UC}{{\rm UC\,}}			
\DeclareMathOperator*{\esssup}{ess\,sup}	
\DeclareMathOperator*{\essinf}{ess\,inf}	

\newcommand{\eps}{\varepsilon}
\newcommand{\ovl}{\overline}

\definecolor{darkred}{rgb}{0.75,0,0.1} 

\theoremstyle{plain}
\newcounter{thm}[section]

\newtheorem{theorem}[thm]{Theorem}
\newtheorem{lemma}[thm]{Lemma}

\newtheorem{corollary}[thm]{Corollary}
\newtheorem{proposition}[thm]{Proposition}
\newtheorem{remark}[thm]{Remark}
\newtheorem{example}[]{Example}

\author[A. Ciomaga]{Adina Ciomaga$^\ddag$}
\address{$^\ddag$ Department of Mathematics, 
	      The University of Chicago, 
	      5734 S. University Avenue Chicago, 
	      IL 60637, USA}
\email{adina@math.uchicago.edu}
 
\author[P.E. Souganidis]{Panagiotis E. Souganidis$^{\sharp,\natural}$}
\address{$^\sharp$ Department of Mathematics, 
	      The University of Chicago, 
	      5734 S. University Avenue Chicago, 
	      IL 60637, USA}
\email{souganidis@math.uchicago.edu}

\author[H.V. Tran]{Hung V. Tran$^\S$}
\address{$^\S$ Department of Mathematics, 
	      The University of Chicago, 
	      5734 S. University Avenue Chicago, 
	      IL 60637, USA}
\email{hung@math.uchicago.edu}
	      
\address{$\natural$ Partially supported by the NSF grants DMS - 0901802 and DMS - 1266383}
\title[Stochastic homogenization of interfaces moving with changing sign velocity]
{Stochastic homogenization of interfaces\\ moving with changing sign velocity}

\graphicspath{{./figures/}{./Bilinear/images/}}

\begin{document}

\begin{abstract}
We are interested in the averaged behavior of interfaces moving in stationary ergodic environments, with  oscillatory  normal velocity which changes sign. This problem can be reformulated, using level sets, as the homogenization of a Hamilton-Jacobi equation with a positively homogeneous non-coercive Hamiltonian. The periodic setting was earlier studied by Cardaliaguet, Lions and Souganidis (2009). Here we concentrate in the random media and show that the solutions of the oscillatory Hamilton-Jacobi equation converge in $L^\infty$-weak $*$ to a linear combination of the initial datum and the solutions of several initial value problems with deterministic effective Hamiltonian(s), determined by the properties of the random media.
\end{abstract}

\maketitle

{\bf Keywords:} 
  stochastic homogenization, 
  Hamilton-Jacobi equations,
  viscosity solutions,
  non-coercive Hamiltonian,
  random media,
  level sets, 		
  front propagation. 	
\smallskip

{\bf AMS Classification:}
  35B27 
  70H20 
  37A50 
  49L25 
  78A48 
\smallskip

\medskip\section{Introduction}\label{sec:Intro}
We study  the averaged behavior of interfaces moving in stationary ergodic environments, with oscillatory  normal velocity which changes sign. The problem can be formulated using the  level-set method,  as the homogenization of a Hamilton-Jacobi equation with a positively homogeneous non-coercive Hamiltonian. For a general overview and further details we refer to Osher and Sethian \cite{OS:88} (see also Souganidis \cite{S:95}, Barles and Souganidis \cite{BS:98} and the references therein). The interface is defined as the zero level set   
$\Gamma^\eps(t,\omega)$ of the solutions $u^\eps$ of the initial value problem
\begin{equation}\label{eq:HJ_eps}
  \begin{cases}
    u^\eps_t + 	a \left (\frac{x}{\eps},\omega\right) |Du^\eps|=0
		& \hbox{ in } \;\;\;\R^n \times  (0,\infty),\\	
    u^\eps=u_0
		& \hbox{ on } \;\;\; \R^n\times\{0\},
  \end{cases}
\end{equation}
where $u_0 \in \UC(\R^n)$, the space of uniformly continuous functions on $\R^n$.  The random environment is modeled by a probability space $(\Omega,\cF,\bP)$ endowed with an ergodic group of measure preserving transformations $(\tau_z)_{z\in\R^n}$. 

The velocity $a=a(y,\omega): \R^n \times \Omega \to \R$  is a stationary random process with respect to $(\Omega, \cF, \bP)$ (precise definitions are given later) and continuous in $y$ for each $\omega$.
The main challenge is to analyze the averaged behavior of the $u^\eps$'s when {$a(\cdot,\omega)$ changes sign}. In this case, the Hamiltonian corresponding to \eqref{eq:HJ_eps}, given by 
  \begin{equation}\label{eq::Hamiltonian}
    H(p,y,\omega)=a(y,\omega)|p|, 
  \end{equation} 
is {neither coercive nor convex} and, hence, the stochastic homogenization of \eqref{eq:HJ_eps} cannot be handled by the theory developed so far in the stationary ergodic environments. 

When $a$ is strictly positive and, for all $\omega\in\Omega$,  $\inf_{\R^n}a(\cdot,\omega)= a_0>0$, the Hamiltonian is convex and coercive. The typical results yield  that, as $\eps\to0$, the $u^\eps$'s converge, locally uniformly in $\R^n\times [0,\infty)$ and a.s. in $\Omega$, to a deterministic $\bu$ which is the solution of the effective initial value problem
\begin{equation}\label{eq:HJ_avg}
  \begin{cases}
    \bu_t + \ovl{H}(D\bu)=0
	  & \hbox{ in } \;\;\;\R^n \times (0,\infty),\\	
    \bu=u_0
	  & \hbox{ on } \;\;\; \R^n\times \{0\}.
  \end{cases}
\end{equation}
The main step is, of course, to determine the  {effective Hamiltonian} $\ovl H = \ovl H(p)$, which is often called the {ergodic constant}. 

The periodic homogenization of coercive Hamilton-Jacobi equations was first studied by Lions, Papanicolaou and Varadhan \cite{LPV} and later by Evans \cite{Ev:89, Ev:92}. Ishii established in \cite{I:00} the homogenization of Hamilton-Jacobi equations in almost periodic settings. The stochastic homogenization for Hamilton-Jacobi equations with convex and coercive Hamiltonians was established independently by Souganidis \cite{S:99} and Rezakhanlou and Tarver \cite{RT:00}. Results for viscous Hamilton-Jacobi equations were obtained by Lions and Souganidis \cite{LS:05} and Kosygina, Rezakhanlou and Varadhan \cite{KRV:06}, while problems with space-time oscillations were considered by  Kosygina and Varadhan \cite{KV:08} and Schwab \cite{S:09}. In \cite{LS:10}  Lions and Souganidis  gave a simpler proof for homogenization in proba\-bility using  weak convergence techniques. Their program was completed by Armstrong and Souganidis in \cite{AS:12,AS:13}, who also introduced {the metric approach}. The viscous 
case was later refined by  Armstrong and Tran in \cite{AT:13}.

When the velocity in \eqref{eq:HJ_eps} changes sign, the Hamiltonian \eqref{eq::Hamiltonian} is neither convex nor coercive. In this case, we consider the connected components $(U_i(\omega))_{i\in I}$ of $\{x\in\R^n : a(x,\omega)\neq 0\}$. In each $U_i(\omega)$, the Hamiltonian is either convex or concave. Therefore, we study the evolution of the $u^\eps$'s in each rescaled connected component $U_{i,\eps}(\omega) : = \left\{\eps x : x\in U_{i}(\omega)\right\}$ and obtain, for each $i\in I$, an effective equation
\begin{equation}\label{eq:HJ_avg-cc}
  \begin{cases}
    \bu_{i,t} + \ovl{H_i}(D\bu_i)=0
	  & \hbox{ in } \;\;\;\R^n \times (0,\infty),\\	
    \bu_i=u_0
	  & \hbox{ on } \;\;\; \R^n \times \{0\},
  \end{cases}
\end{equation}
with a {deterministic} effective Hamiltonian $\ovl{H_i}=\ovl{H_i}(p)$. It is, of course, necessary to assume that the sets $U_i(\omega)$  are themselves stationary ergodic (definitions are given later). The result is that, as $\eps\to 0$,  the $u^\eps$'s converge locally uniformly in $U_{i,\eps}(\omega) \times [0,\infty)$ to the solution $\bu_i$ of \eqref{eq:HJ_avg-cc} and locally in $L^{\infty}(\R^n\times(0,\infty))$-weak $\star$ to
\begin{equation}\label{eq:weak-conv}
   \bu:=\theta_0 u_0 + \sum_{i\in I} \theta_i \bu_i,
\end{equation}
where, for all $i\in I$, $\theta_i:= \bE[{\bf 1}_{U_i(\omega)}]$ and 
$\theta_0 := \bE[{\bf 1}_{\{a(\cdot,\omega) =0\}}]$.

This result was already established in the periodic setting by Cardaliaguet, Lions and Souganidis \cite{CLS:09}. We recall that, in the periodic case, each effective Hamiltonian $\ovl{H_i}(p)$ is found by solving the associated cell problem in the connected component $U_i$, that is, for each $p\in\R^n$, there is a unique constant $\ovl{H_i}(p)$  such that there exists a periodic solution $w$,  called {corrector}, of
\begin{equation}\label{eq:HJ_cell-cc}
a(y)|p + Dw| = \ovl{H_i}(p)\;\;\; \hbox{ in } \;\;\;{U_i}.
\end{equation} 

The main difficulty in  random environments is {the lack of correctors}. When the media is sta\-tionary ergodic,  one can establish existence of subcorrectors for convex and coercive Hamiltonians \cite{LS:05}, that is, subsolutions of  \eqref{eq:HJ_cell-cc} with strictly sublinear decay at infinity. To overcome the lack of correctors, a new approach was initiated in  \cite{AS:12, AS:13}  to study the stochastic homogenization of Hamilton-Jacobi equations by introducing the metric problem. 

We follow here this approach and consider, for $\mu>0$ and each unbounded connected component $U_i(\omega)$ of $\{x\in\R^n:a(x,\omega)>0\}$ (respectively for $\mu<0$ and each unbounded connected component $U_i(\omega)$ of $\{x\in\R^n:a(x,\omega)<0\}$) and  $z\in U_i(\omega)$, 
\begin{equation}\label{eq:metric-pb-intro}
  \begin{cases}
      a(y,\omega) |Dm_\mu| = \mu & \hbox{ in }\ U_i(\omega) \setminus \{z\},\\
      m_\mu(\cdot,z,\omega)= 0& \hbox{ at } \{z\}.
  \end{cases}
\end{equation}
Problem \eqref{eq:metric-pb-intro} has a stationary and  subadditive  maximal subsolution (respectively superadditive minimal supersolution) $m_\mu$ in $U_i(\omega)$, which can be thought of as the minimal travel time from $z$ to $y$, with passage velocity $a(\cdot,\omega)$, while staying within the connected domain $U_i(\omega)$. The difficulty is that the $m_\mu$'s are defined only in $U_i(\omega)$ and not in $\R^n$. 

The lack of coercivity of \eqref{eq::Hamiltonian} creates an additional problem. Since $a(\cdot,\omega) = 0$ on each $\partial U_i(\omega)$, the $m_\mu$'s are not Lipschitz continuous up to the boundary of $U_i(\omega)$. In fact, as we show, the $m_\mu$'s blow up on $\partial U_i(\omega)$ with a logarithmic rate.  In general, Lipschitz estimates for solutions of Hamilton-Jacobi equations  are given in terms of the geodesic distance within the domain (see for example Lions \cite{PLL:82}). Here, using the local structure of the equation, we give sharp Lipschitz bounds of the $m_\mu$'s, {independent of the geodesic distance}. As a consequence, homogenization holds for a  wide class of domains, such as classical percolation structures from probability (see for example Garet and Marchand  \cite{GM:04}, Antal and Pisztora  \cite{AP:96}) and Poisson point environments.

The tool to study the behavior of the solution to the metric problem is the subadditive ergodic theorem, which, however, cannot be applied directly to the $m_\mu$'s, since they are not defined every\-where. To overcome this drawback, we first extend $m_\mu$ on each ray starting from the origin,  apply the subadditive ergodic theorem to the extension and then relate back the averaged limit to $m_\mu$. This difficulty is amplified by the blow up of the $m_\mu$'s on $\partial U_i(\omega)$. To overcome the latter, we restrict the $m_\mu$'s  to the sets \[U_i^\delta(\omega) = \{x\in U_i(\omega): |a(x,\omega)|>\delta\}.\]
We show that the averaging takes place in $U^\delta_i(\omega)$ and the limit is independent of $\delta$. Moreover, in view of the Lipschitz estimates, we obtain an averaging of $m_\mu$, in the liminf, up to the boundary. More precisely, we show that, for each connected component of $\{x\in\R^n : a(x,\omega)>0\}$ 
\[
\liminf_{\substack{t \to \infty\\t y\in U_i(\omega)}} \frac{m_\mu(ty,0,\omega)}{t} = \lim_{\substack{t \to \infty\\t y\in U_i^\delta(\omega)}} \frac{m_\mu(ty,0,\omega)}{t}=\ovl{m}_\mu(y).\]
When we consider connected components corresponding to $\{x\in\R^n : a(x,\omega)<0\}$ the liminf becomes limsup.

The effective Hamiltonian is then characterized by the $m_\mu$'s as the smallest (respectively the largest) constant for which the metric problem \eqref{eq:metric-pb-intro} has a subsolution with strictly sublinear decay at infinity. We provide an inf-sup representation formula for each deterministic effective Hamiltonian $\ovl{H_i}(p)$ and establish the homogenization result in each connected component. The weak convergence  \eqref{eq:weak-conv} then follows.

Few results are available for non-coercive Hamiltonians and they all rely on some reduction property that compensates for the lack of coercivity; see, for example, Alvarez and Bardi \cite{AB:10}, Barles \cite{B:07} and Imbert and  Monneau \cite{IM:08}. A different approach, based on nonresonance conditions, was initiated by Arisawa and Lions \cite{AL:98} and extended to periodic noncoercive-nonconvex Hamiltonians by Cardaliaguet in \cite{C:10}. Of special  interest is the study of noncoercive Hamilton-Jacobi equations associated to moving interfaces. The homogenization of \eqref{eq:HJ_eps} in the periodic setting was established by Cardaliaguet, Lions and Souganidis \cite{CLS:09}. Prior to their work,  Craciun and Bhattacharya  \cite{CB:03} discussed formally and gave numerical examples for periodic homogenization of such problems. The homogenization of the G-equation (used as model for fronts propagating with normal velocity and advection)  was recently established by Cardaliaguet, Nolen and Souganidis \
cite{CNS:11} in spatio-temporal periodic environments and  by Cardaliaguet and Souganidis in \cite{CS:13} in random media. Homogenization of general non-convex Hamiltonians in random environments has remained, until now, completely open. A first extension to level-set convex Hamiltonians was proven by Armstrong and Souganidis in \cite{AS:13} and, more recently, Armstrong, Tran and Yu established stochastic homogenization for double-well type Hamiltonians in \cite{ATY:13}. 

\subsection*{The organization of the paper} 
In Section 2 we  introduce the assumptions on the random media and the Hamiltonian and state the main results. In Section  3 we discuss some examples. We study the metric problem and give optimal Lipschitz regularity results of the maximal subsolution in Section 4. We prove in Section 5 the homogenization of the metric problem and obtain a deterministic effective Hamiltonian in Section 6. The main homogenization result is shown in Section 7.

\subsection*{Notations}
We work in the $n$-dimensional Euclidean space $\R^n$ and we denote by $\Q^n$ and $\Z^n$  respectively the sets of points with rational and integer coordinates.  For a unit vector $y\in\R^n$, we write $\R^+y := \{ty : t\geq 0 \}$. We denote by $B_r(x_0)$ the ball of radius $r$ centered at $x_0$ and by $B_R$ the ball of radius $R$ centered at the origin.  If $U\subset\R^n$, we denote by ${\bf 1}_U$ the characteristic function of $U$. For all $\eps>0$ we consider the rescaled sets $U_\eps = \{\eps x : x\in U\}$. Let $\mathcal C^1(U)$ and $\mathcal C_c^\infty(\R^n)$ be respectively the sets of continuously differentiable functions on $U$ and of infinitely differentiable functions on $\R^n$ with compact support. We denote by $\cL$ the class of Lipschitz functions on $\R^n$.

\medskip\section{Assumptions and Main Results}\label{sec::Assumptions}
We describe here the general setting,  introduce the  assumptions and  present the main theorems. For the reader's convenience, we recall some preliminary results to be used later in the paper. 

\subsection*{The general random setting}

We consider a probability space $(\Omega, \cF, \bP)$ endowed with an {ergodic group of measure preserving transformations} $(\tau_z)_{z\in\R^n}$, that is, a family of maps $\tau_z:\Omega\to\Omega$ satisfying,  for all $z,z'\in \R^n$ and all $\cU\in\cF$,
	\[\tau_{z+z'}=\tau_{z}\circ\tau_{z'}\;\; \hbox{ and }\;\;
	\bP[\tau_z\cU] = \bP[\cU]\]
and 
	\[\hbox{ if }\tau_z(\cU)=\cU \hbox{ for every } z\in\R^n,
	\hbox{ then either }\bP[\cU] = 1\hbox{ or }\bP[\cU]=0.\]
 
\subsection*{The assumptions}
 
Let $a:\R^n\times\Omega\to\R$ be measurable with respect to the $\sigma$-algebra generated  by $\cB\times\cF$, where  $\cB$ denotes the Borel $\sigma$-algebra of $\R^n$ and assume that

\begin{enumerate}
 \item[(A1)] $a$ is {stationary with respect to 
 	$(\tau_z)_{z\in\R^n}$}, that is, for every $y,z\in \R^n$ and each $\omega\in\Omega$, 
	\[a(y,\tau_z\omega) = a(y+z,\omega),\]    
 \item[(A2)] the family $\big(a(\cdot,\omega)\big)_{\omega\in\Omega}$ is 
	{equi-bounded and equi-Lipschitz continuous, with Lipschitz constant $L>0$}, that is,
	for every $y,z\in \R^n$ and $\omega\in\Omega$, 
	\[|a(y,\omega) - a(z,\omega)|\leq L |y-z|,\]
 \item[(A3)] for each $\omega\in\Omega$, $a(\cdot,\omega)$ {changes sign}. 
\end{enumerate}
Let $\big(U_i(\omega)\big)_{i\in I^+}$ be the connected components of $\{a(\cdot,\omega)>0\}$ and $\big(U_i(\omega)\big)_{i\in I^-}$ be the connected components of $\{a(\cdot,\omega)<0\}$ such that, for all $i\in I^\pm$, 
\[\partial U_i(\omega)\subset U_0 (\omega) : =  \{x\in\R^n : a(x,\omega) = 0 \}.\]
Denote by $I := I^+\cup U^-$ and $I_0 := I \cup \{0\}$.

We introduce next the assumptions on the random structure of $\big(U_i(\omega)\big)_{i\in I_0}$. At first look, the conditions appear uncheckable since there is no topology in $\Omega$, but later in the paper (Section \ref{sec::examples}) we give several examples of random media that satisfy these assumptions. We assume that:
\begin{enumerate}
 \item[(S1)] there exists a family $\{\cU_i\}_{i\in I_0}$ of measurable subsets of $\Omega$ 
	such that $\bP[\cU_0]=\theta_0\geq 0$ and, for all $i \in I$, $\bP[\cU_i] = \theta_i>0$ and, for each $\omega\in\Omega$,  $U_i(\omega)$ is {stationary}, that is, 
	\[U_i(\omega) = \big\{x\in\R^n \,:\, \tau_x \omega\in \cU_i\big\},\]
 \item[(S2)] for each $\omega\in\Omega$ and all $i \in I^\pm$, $U_i(\omega)$ is
	{$\delta$-connected}, that is, there exists $\delta_0>0$ such that, for every $\delta\in(0,\delta_0)$, $U^{\delta}_i(\omega)$ is connected and
\item[(S3)] if $U_i(\omega)$ with $i\in I$ is unbounded, then, for every $\delta\in(0,\delta_0)$ and every unit vector $y\in \R^{n}$,  there exists an (at most) countable family of random intervals $\big( (s_j(\omega),t_j(\omega))\big)_{j=1}^\infty$, which depend on $y$ and $\delta$, so that
	\[I^\delta_y(\omega): = \{ t\geq 0: ty\not\in U^{\delta}_i(\omega)\} 
	= \bigcup_{j=1}^\infty (s_j(\omega),t_j(\omega)) \]
	and, uniformly in $y$ and $\delta$, 
	\begin{equation}\label{eq:intervals}
	 \lim_{j\to\infty}\frac{s_j(\omega)}{t_j(\omega)} = 1 \hbox{ a.s. in }\omega.
	\end{equation}
\end{enumerate}

In view of Borel-Cantelli Lemma, checking \eqref{eq:intervals} is equivalent to showing that, for every $\eps>0$,
      \[\sum_{j=1}^\infty\bP\left[\left\{ \omega\in\Omega : \left|\frac{t_j(\omega)-s_j(\omega)}{t_j(\omega)}\right|>\eps\right\}\right]<\infty.\]

\subsection*{The main result} 
We now present the homogenization result, which is established in several steps. First we show that, for each connected component $U_i(\omega)$ with $i\in I$, there exists a deterministic effective Hamiltonian and then we prove that, as $\eps\to 0$,  the $u^\eps$'s converge locally uniformly in each $U_{i,\eps}(\omega)$ to a deterministic $\bar u_i$. The $L^\infty$-weak $\star$ convergence of the $u^\eps$'s to an averaged profile $\bar u$ follows.

For each $\mu > 0$, $\omega\in\Omega$ and $U(\omega)$ a connected component of $\{x\in\R^n: a(x,\omega)\neq0\}$, let ${\cS_\mu(U(\omega))}$ be the class of functions  $w$ such that, for each ${\eta}>0$ sufficiently small, there exists $\ovl{\delta}=\ovl{\delta}({\eta})>0$, which is independent of $\omega$, satisfying $\lim_{{\eta} \to 0}\ovl{\delta}({\eta})=0$ and, for all $\delta \in (0,\ovl{\delta})$ and all $y_1,y_2 \in U^{\delta}(\omega)$,
\[
|w(y_1)-w(y_2)| \leq \left( \frac{\mu}{{\delta}}+{\eta}\right)|y_1-y_2|.
\]
Define $\cL_\mu(U(\omega)) := \cL\cap \cS_\mu(U(\omega)).$ We also consider the classes $\cS^\pm$ of Lipschitz continuous functions which are strictly sublinear from below and  from above respectively, that is
\begin{eqnarray*}
 \cS^+ :=\Big\{ w\in \cL : \  \liminf_{|y| \to \infty}\frac{w(y)}{|y|} \geq 0 \Big\}\hbox{ and }
 \cS^- :=\Big\{ w\in \cL : \  \limsup_{|y| \to \infty}\frac{w(y)}{|y|} \leq 0 \Big\}.
\end{eqnarray*}
 
\begin{theorem}[The effective Hamiltonians]\label{thm:effective_Hamiltonian}
Assume $(A1),(A2),(A3)$, $(S1),(S2)$ and $(S3)$. For each $i\in I$ there exists a deterministic, 1-positively homogeneous, continuous effective Hamiltonian $\ovl{H}_i:\R^n\to\R$ such that\\
 (i) if $U_i(\omega)$ is bounded, then $\ovl H_i\equiv0$,\\
(ii) if $U_i(\omega)$  is unbounded and $i\in I^+$, then  $\ovl{H}_i$ is convex, nonnegative
      and, for all $p\in\R^n$,
      \begin{equation}\label{eq:eff-H-positive}
      \ovl{H}_i(p) : = \inf_{\substack{(w,\lambda) \in {\cS^+}\times(0,\infty)\\w+p\cdot y\in{\cL_\lambda(U(\omega))}}} \left[\esssup_{y\in U_i(\omega)} 
      \Big( a(y,\omega)|p+Dw|\Big)\right],
      \end{equation}
(iii) if $U_i(\omega)$ is unbounded and $i\in I^-$, then $\ovl{H}_i$ is concave,      nonpositive and, for all $p\in\R^n$,
      \begin{equation}\label{eq:eff-H-negative}
      \ovl{H}_i(p) := \sup_{\substack{(w,\lambda) \in {\cS^-}\times(0,\infty)\\w+p\cdot y\in{\cL_\lambda(U(\omega))}}}\left[\essinf_{y\in U_i(\omega)} 
      \Big( a(y,\omega)|p+Dw| \Big)\right].
      \end{equation}
\end{theorem}
For each $i\in I$ we solve the averaged equation
 \begin{equation}\label{eq:HJs_avg}
  \begin{cases}
    \bu_{i,t} + \ovl{H}_i(D\bu_i)=0	&\hbox{ in } \;\;\;\R^n \times (0,\infty),\\	
    \bu_i =u_0 				&\hbox{ on } \;\;\; \R^n \times\{0\}
  \end{cases}
\end{equation}
and establish the following homogenization result.  

\begin{theorem}[Homogenization]\label{thm:main_result}\label{thm::homog-time}
Assume $(A1),(A2),(A3)$, $(S1),(S2)$ and $(S3)$. There exists $\ovl\delta_0>0$ and an event of full probability  $\widetilde \Omega\subseteq\Omega$ such that, for each $\omega\in\widetilde\Omega$, $\delta\in(0,\ovl \delta_0)$ and $T>0$, the unique solution $u^\eps = u^\eps(\cdot,\cdot,\omega)$ of \eqref{eq:HJ_eps} converges locally uniformly in $U^\delta_{i,\eps}(\omega)  \times [0,T]$ to the unique solution $\bu_i$ of \eqref{eq:HJs_avg}, that is, for $(x,t)\in \R^n\times (0,T)$,
\begin{eqnarray}
\label{eq:conv-cc+}
  	\liminf_{\substack{\eps\to 0\\ y\to x, s\to t \\ y\in U_{i,\eps}(\omega)} }u^\eps (y,t,\omega) =
   	\limsup_{\substack{\eps\to 0\\ y\to x, s\to t \\ y\in U^\delta_{i,\eps}(\omega)} }u^\eps (y,t,\omega) = 
	\bar u_i(x,t)& \hbox{ if } \;\;\; i\in I^+, \\
\label{eq:conv-cc-}
	\limsup_{\substack{\eps\to 0\\ y\to x, s\to t \\ y\in U_{i,\eps}(\omega)} }u^\eps (y,t,\omega) =
   	\liminf_{\substack{\eps\to 0\\ y\to x, s\to t \\ y\in U^\delta_{i,\eps}(\omega)} }u^\eps (y,t,\omega) = 
	\bar u_i(x,t) & \hbox{ if }\;\;\; i\in I^-. \\
\end{eqnarray}
Moreover, for each $\omega\in\tilde\Omega$ and $R>0$, as $\eps\to0$,
\begin{equation}
  u^\eps \stackrel{*}\rightharpoonup \bu:=\theta_0 u_0 + \sum_{i\in I} \theta_i \bu_i
  \ \hbox{ in} \ L^\infty(B_R \times (0,T)).
\end{equation}
\end{theorem}

\subsection*{Some preliminary results}
An important tool in our analysis is the subadditive ergodic theorem for continuous subadditive processes. For the reader's convenience we give below the precise statement and we refer to Akcoglu and Krengel \cite{AK:81} for its proof.

Let $(\sigma_t)_{t\ge 0}$  and $\mathcal I$ be respectively a semi-group of measure-preserving transformations on  $\Omega$ and the class of subsets of $[0,\infty)$ which are finite unions of intervals of the form $[a,b)$. A {continuous subadditive process} on $(\Omega,\cF,\bP)$ is a map  $Q:\mathcal{I} \to L^1(\Omega,\bP)$ which is
\begin{enumerate}[(i)]
  \item stationary invariant, that is,
	      $ Q(I)(\sigma_t\omega) = Q(t+I)(\omega)
	      \hbox{ for all } t>0, I\in\mathcal{I}, \hbox{ a.s. in } \omega$,
  \item uniformly integrable, that is, there exists $C>0$ such that,
	      $\bE|Q(I)| \leq  C |I| \hbox{ for all } I\in \mathcal{I},$ and
  \item subadditive with respect to unions of disjoint intervals, that is, if $I_1,I_2,..., I_k\in\mathcal I$ are disjoint, then
	      $Q(\bigcup_{i=1}^k I_i)\leq \sum_{i=1}^k Q(I_i).$
\end{enumerate}
 
\begin{theorem}[The subadditive ergodic theorem]\label{thm:SED}
Let $Q:\mathcal{I} \to L^1(\Omega,\bP)$ be a continuous subadditive process on $(\Omega,\cF,\bP)$ with respect   to the semi-group of measure-preserving transformations  $(\sigma_t)_{t\ge 0}$. Then there exists a random variable $q$ such that 
 \begin{equation*}
    \lim_{t\rightarrow \infty} \frac{1}{t} Q\big([0,t)\big)(\omega) = q(\omega) 
    \hbox{ a.s. in }\omega.
 \end{equation*}
If, in addition,  $(\sigma_t)_{t\ge 0}$ is ergodic, then $q(\omega)\equiv \bar q$ a.s., for some constant $\bar q$.
\end{theorem}

Solutions of the Hamilton-Jacobi equations are to be understood in the viscosity sense (see, for example, the ``User's Guide'' of Crandall, Ishii and Lions \cite{CIL:92} and the books by Barles \cite{B:94} and Bardi and Capuzzo-Dolcetta \cite{BCD:97}). The following well known fact will be used several times in the paper.

\begin{lemma}\label{lem:equiv-soln}
Let $p\in\R^n$, $\mu\in\R$, $U\subset\R^n$ open and $a:U\to\R$ such that $a>0$ in $U$ and $a=0$ on $\partial U$. Then $w\in \cL$ is a  viscosity solution of 
 \begin{equation}\label{eq:susol}
  a(y) |p+Dw| \leq \mu \;\;\;\hbox{ in } \;\;\;U
 \end{equation}
if and only if $w$ satisfies \eqref{eq:susol} almost everywhere in $U$.
\end{lemma}

\medskip\section{Examples of stationary random environments satisfying $(S1),(S2)$ and $(S3)$}\label{sec::examples}

We present several examples of random media, some occurring in percolation theory, that satisfy our assumptions.  

\begin{example}[\em Site percolation]\label{ex::percolation}\rm
We briefly review a site percolation structure in $\R^n$. The space is  divided into cubes of size $1$ with vertices on $\Z^n$ and each cube is painted, independently, white or black with probability $p$ or $1-p$ respectively, where $0\leq p\leq1$. Cubes are always assumed to be closed, unless otherwise stated.

For each $\omega\in\Omega$, let $F(\omega)$  be the union of all black cubes. The interior of $F(\omega)$ and $\R^n\setminus F(\omega)$ consist of mutually disjoint open connected components that we denote  respectively  by $\big(U_i(\omega)\big)_{i\in I^-}$ and $\big(U_i(\omega)\big)_{i\in I^+}$.  It is known (see Grimmett \cite{Gr:99}) that there exists a critical threshold $p_c>0$ such that if $p>p_c$, then a.s. there only exists one unbounded connected component $U_\infty(\omega)$ of $\R^n\setminus F(\omega)$. Figure \ref{Fig:percolation} illustrates a site percolation model where $U_\infty(\omega)$ is painted in blue. Let $\Pi^n$ be the set of unit cubes whose vertices are points of $\Z^n$. The probability space $(\Omega,\cF, \bP)$ is $\Omega := \{0,1\}^{\Pi^n}$, $\cF =\{\cU:\cU\subset\Omega\}$ and
$\bP=\left(p\delta_1 + (1-p)\delta_0\right)^{\otimes\Pi^n}.$ 
The translation operators $(\tau_z)_{z\in\Z^n}$, given by $\tau_z :\Omega\to\Omega$ with $(\tau_z \omega) (\pi) : = \omega(z+\pi), \hbox{ for all }\pi\in \Pi^n,$ are ergodic.
\begin{figure}[!ht]
    \centering\includegraphics[width=0.3\linewidth]{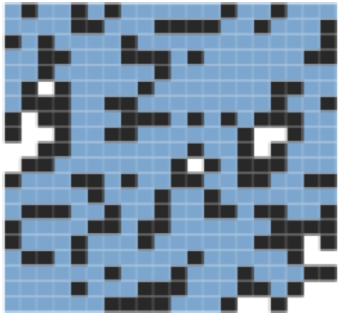} 
    \caption{Bernoulli percolation.}
    \label{Fig:percolation}
\end{figure}

Theorem \ref{thm::homog-time} holds for any velocity $a(\cdot,\omega)$ such that $a(\cdot,\omega)>0$ in $U_i(\omega)$ with $i\in I^+$,  $a(\cdot,\omega)<0$ in $U_i(\omega)$ with $i\in I^-$ and $a(\cdot,\omega)=0$ on $\partial U_i(\omega)$ with $i\in I^\pm$, which is bounded, Lipschitz continuous and stationary with respect to $(\tau_z)_{z\in\Z^n}$. A particular choice of $a(\cdot,\omega)$ is the distance function to the boundaries of $U_i(\omega)$, that is,
\[a(x,\omega) = 
      \begin{cases}
          \min\big( d(x,\partial U_i(\omega)),1/2 \big) & \hbox{ if }x\in U_i(\omega), \ \text{with} \ i \in I^+,\\
         -\min\big( d(x,\partial U_i(\omega)),1/2 \big) & \hbox{ if }x\in U_i(\omega), \ \text{with} \ i \in I^-.
      \end{cases}
\]
Clearly $a$ satisfies $(A1),(A2),(A3)$ and $(S1),(S2)$ hold for all $U_i(\omega)$ with $i\in I_0$, while  $U_\infty(\omega)$ satisfies $(S3)$. Indeed, for any $y\in\R^n$ with $|y|=1$ and $\eps>0$, we have
\begin{eqnarray*}
    \bP\left[\left\{\omega\in\Omega : \left|\frac{t_j(\omega)-s_j(\omega)}{t_j(\omega)}\right|>\eps\right\}\right] 
    & = &
    \bP\left[\left\{\omega\in\Omega : \left|\frac{t_1(\omega)-s_1(\omega)}{t_j(\omega)}\right|>\eps\right\}\right]\\ 
    & \leq &
    \bP\Big[\left\{\omega\in\Omega : \left|t_1(\omega)-s_1(\omega)\right|>j\eps\right\}\Big] \leq p^{j\eps}
\end{eqnarray*}
and, thus,
\begin{eqnarray*}
    \sum_{j=1}^\infty \bP\left[\left\{\omega\in\Omega : \left|\frac{t_j(\omega)-s_j(\omega)}{t_j(\omega)}\right|>\eps\right\}\right] 
    \leq \sum_{j=1}^{\infty} p^{\eps j} = \frac{1}{1-p^\eps}.
\end{eqnarray*}
In view of Theorem \ref{thm:effective_Hamiltonian} all effective Hamiltonians are null, except for $\ovl H_\infty$ corresponding to $U_\infty(\omega)$. 
Theorem \ref{thm::homog-time} then yields, the $L^\infty$-weak $\star$ convergence, as $\eps\to0$, of the $u^\eps$'s  to 
$\ovl u = (1-\theta)u_0 + \theta \ovl u_\infty$
with $\theta= \bE [{\bf 1}_{U_\infty(\omega)}]$, where $\ovl u_\infty$ is the solution of \eqref{eq:HJs_avg} corresponding to $\ovl H_\infty$.

In the particular case of a periodic checkerboard configuration and of a periodic velocity $a$, all $H_i\equiv0$  and, hence, $u^\eps$ converges uniformly to the initial data $u_0$. We recover thus the result of  \cite{CLS:09}, also known as {trapping}.
\end{example}

\begin{example}[\it Site percolation on regular latices]\rm
Theorem \ref{thm::homog-time} holds for general percolation structures given by regular lattices, such as those displayed in Figure \ref{Fig:percolation2}, where as before, each cell of the lattice is painted white or black with probability $p$ and $1-p$  respectively. 
\begin{figure}[!ht]
    \centering\includegraphics[width=0.85\linewidth]{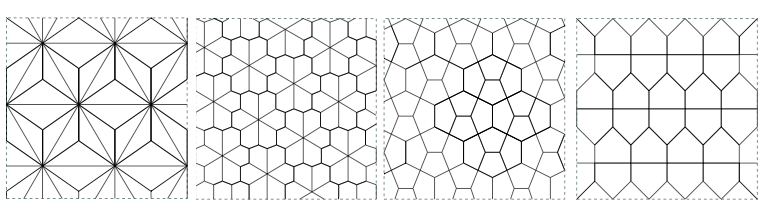} 
    \caption{Percolation on Archimedean Lattices (all polygons are regular and each vertex is surrounded by the same sequence of polygons).}
    \label{Fig:percolation2}
\end{figure}
\end{example}

\begin{example}[\it Site percolation with isolated obstacles]\rm
Consider a site percolation $\Pi^n$ as in Example \ref{ex::percolation}, where now each cube $\pi\in\Pi^n$ is either empty with probability $p$ or compactly contains a ball of fixed size  with probability $1-p$. Let $a$ be a Lipschitz continuous, bounded and translation invariant velocity, which is negative inside each ball and non-negative outside. Then there exists one infinite connected component $U_\infty(\omega)$ with effective Hamiltonian $\ovl{H}_\infty>0$ and, in view of Theorem \ref{thm::homog-time}, as $\eps\to0$, the  $u^\eps$'s converge weakly to $\ovl u = (1-\theta)u_0 + \theta \ovl u_\infty, $
where $\theta= \bE[{\bf 1}_{U_\infty(\omega)}]$ and $\ovl u_\infty$ is the solution  of \eqref{eq:HJs_avg} associated to  $U_\infty(\omega)$. In particular, in the periodic setting, this construction yields a periodic configuration of small holes and we recover the result of  \cite{CLS:09}.
\end{example}

\begin{example}[\it Poisson cloud]\rm
Theorem \ref{thm::homog-time} applies to Poisson environments. Consider a random set $F(\omega) = \bigcup_{i=1}^\infty \ovl B_{r_i}(x_i,\omega)$, where the points $(x_i)_{i=1}^\infty$ are given by a Poisson point process and the radii  $(r_i)_{i=1}^\infty$ are  independent and identically distributed on $(0,\infty)$ (see Figure \ref{Fig:poisson}). Let $(U_i(\omega))_{i\in I^+}$ and $(U_i(\omega))_{i\in I^-}$ be respectively the connected components of $\R^n\setminus F(\omega)$ and the interior of $F(\omega)$ and $a$ be a Lipschitz continuous, bounded and translation invariant velocity which is negative inside each ball and positive outside. One can easily check,  as in Example \ref{ex::percolation}, that  assumptions $(S1),(S2)$ and $(S3)$ are satisfied.
 \begin{figure}[!ht]
    \centering\includegraphics[width=0.4\linewidth]{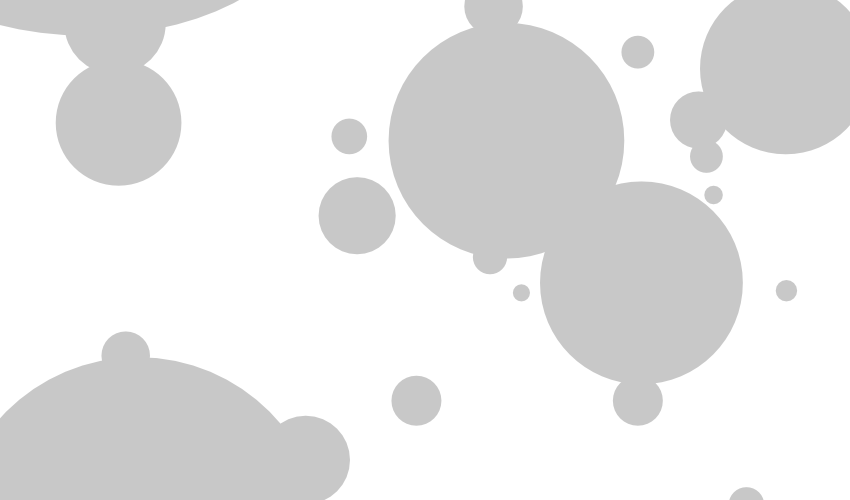} 
    \caption{Homogenization holds for Poisson point process.}
    \label{Fig:poisson}   
\end{figure}
\end{example}

\medskip\section{The Metric Problem}\label{sec::metric-pb}

Fix $\omega\in\Omega$ and assume that $U(\omega)$ is an unbounded connected component of $\{x\in\R^n  : a(x,\omega)>0\}$ with  $\partial U(\omega) \subseteq \{x\in\R^n : a(x,\omega) = 0\}$. For each fixed constant $\mu\geq 0$ and each $z\in U(\omega)$ we consider the eikonal-type problem
\begin{equation}\label{eq:metric-pb}
  \begin{cases}
  	a(y,\omega) |Dm_\mu| = \mu & \hbox{ in }\ U(\omega) \setminus \{z\},\\
  	m_\mu(\cdot,z,\omega)= 0& \hbox{ at } \{z\}.
  \end{cases}
\end{equation}
We show below that there exists a maximal, positive, subadditive, locally Lipschitz continuous subsolution, which blows-up on the boundary of $U(\omega)$. Similar arguments remain true for the unbounded connected components of $\{x\in\R^n : a(x,\omega)<0\}$, in which case \eqref{eq:metric-pb}  is well posed for each fixed constant $\mu\leq 0$  and has instead a minimal, negative, superadditive supersolution.

\subsection*{Maximal subsolutions}
We first define the maximal subsolution of \eqref{eq:metric-pb} and establish some basic properties.\smallskip

For each $\mu \ge 0$ and $z\in U(\omega)$ let $m_\mu(\cdot,z,\omega):U(\omega)\to[0,\infty)$ be given by
\begin{eqnarray}\label{eq:def-max-subsol}
    	m_\mu(y,z,\omega) : = 
	\sup \big \{ w(y,\omega): w(\cdot,\omega) \in \cL_\mu(U(\omega)), w(z,\omega)=0,\;
	a(y,\omega)|Dw| \le \mu \ \text{in} \ U(\omega)  \big \}.
\end{eqnarray}
Recall that, in view of Lemma \ref{lem:equiv-soln}, the differential inequality in \eqref{eq:def-max-subsol} may be interpreted either in the viscosity sense or in the almost everywhere sense.

Note that $m_\mu(\cdot,z,\omega)$ is well defined in $U(\omega)$ for all $\mu\geq 0$. Indeed, the set of admissible $w$ in \eqref{eq:def-max-subsol} is nonempty, since $w\equiv 0$ satisfies all the requirements, a fact which also implies that $m_\mu(\cdot,z,\omega) \geq 0$. Moreover, in view of the Lipschitz bounds we establish below, $m_\mu(y,z,\omega)$ is finite for all $y\in U(\omega)$. Also, observe that for $\mu=0$, $m_\mu(\cdot,z,\omega)\equiv 0$ in $U(\omega)$. Henceforth we only deal with $\mu>0$.

The next lemma summarizes two of the most important properties of $m_\mu$, namely that is maxi-mal and locally Lipschitz continuous in $U(\omega)$. 

\begin{lemma}\label{lemma::Lipschitz}
Assume $(S2)$. For each $z\in U(\omega)$, $m_\mu(\cdot,z,\omega)\in \cS_\mu(U(\omega))$ 
and it is a maximal subsolution of \eqref{eq:metric-pb}, that is,  if $w\in\cL_\mu(U(\omega))$ is a subsolution of $ a(y,\omega)|Dw|\le\mu$ in $U(\omega)$, then $w(\cdot)-w(z)\leq m_\mu(\cdot,z,\omega)$.
\end{lemma}

\begin{proof}
The result is a direct consequence of the definition of $\cL_\mu(U(\omega))$ and $(S2)$. The fact that $m_\mu(\cdot,z,\omega)$ is a subsolution of \eqref{eq:metric-pb} is immediate, since it is the supremum of a collection of subsolutions (see \cite{CIL:92}.)
\end{proof}

Since $a(\cdot,\omega)$ is degenerate on $\partial U(\omega)$, the Hamiltonian is not coercive and thus the solutions are not Lipschitz up to the boundary. We show next that the $m_\mu$'s blow up on the boundary of $U(\omega)$ with a logarithmic rate.
\begin{proposition}[Blow-up on the boundary]\label{prop::blow-up}
Assume $(A2)$ and $(S2)$. For all $y,z\in U(\omega)$, 
\begin{equation*}
    	m_\mu(y,z,\omega) \geq  \frac{\mu}{{L}} \left|
    	\log{a(y,\omega)}-\log{a(z,\omega)}\right|
\end{equation*}
and, for each fixed $z\in U(\omega)$, 
	\[\lim_{\substack{{\hbox{dist}(y,\partial U(\omega))}\to 0, \; y\in U(\omega)}} m_\mu (y,z,\omega) = +\infty. \]
\end{proposition}

\begin{proof}
Fix $z\in U(\omega)$ and  $\delta\in(0,\min(\delta_0, a(z,\omega)/2))$ with $\delta_0>0$ given by $(S2)$. Recall that in view of $(S2)$, $U^\delta(\omega) = \{x \in U(\omega): a(x,\omega)>\delta \}$ is connected and note that $z\in U^\delta(\omega)$. Define the barrier  
\begin{equation*}
    	b^\delta(y,\omega)=
      	\begin{cases}
      		-\frac{\mu}{L}  \log a(y,\omega) +c & \hbox{ in } \;\;\;U^\delta(\omega),\\
      		-\frac{\mu}{L} \log \delta + c& \hbox{ in } \;\;\;\R^n\setminus U^\delta(\omega),
      	\end{cases}
\end{equation*}
where  $c= (\mu/{L}) \log a(z,\omega)$. Since $a(\cdot,\omega)$ is differentiable a.e. in $U^\delta(\omega)$, 
\begin{equation*}
    	\|Db^\delta(\cdot,\omega)\|_{L^\infty(U^\delta(\omega))}
    	\; \leq \; \frac{\mu}{{L}} \frac{\|Da(\cdot,\omega)\|_{L^\infty}}
    	{\inf_{U^\delta(\omega)} a(\cdot,\omega)}
    	\; = \; \frac{\mu}{\delta}.
\end{equation*}
Thus $b^\delta(\cdot,\omega) \in{\cL_\mu}(U(\omega))$ with a global Lipschitz constant equal to $\mu/\delta$, $b^\delta(z,\omega)=0$ and satisfies  
\[a(y,\omega)|D b^\delta| = 0 \;\;\; \hbox{ in } \; U(\omega)\setminus U^\delta(\omega).\]
Therefore
\[a(y,\omega)|D b^\delta| \leq \mu \;\;\; \hbox{ in } \;U^\delta(\omega)\]
and, hence, $b^\delta$ is an admissible subsolution in \eqref{eq:def-max-subsol}. In view of the maximality of $m_\mu$, for all $y\in U(\omega)$,
	\[  m_\mu(y,z,\omega) \geq \sup_{\delta} \left(b^\delta(y,\omega)\right) =
    	-\frac\mu{{L}} \log a(y,\omega) + c.\] 
The claimed bound is now obvious from the positivity of $m_\mu(\cdot,\cdot,\omega)$ and the definition of $c$. 

We next remark that the Lipschitz continuity of $a(\cdot,\omega)$ and the fact that $a(\cdot,\omega) = 0$ on $ \partial U(\omega)$ imply, for all $y\in U(\omega)$ and $d(\cdot,\omega):= \text{dist}(\cdot,\partial U(\omega))$,
	\[ 0\leq a(y,\omega)  \;\leq\; \|Da(\cdot,\omega)\|_{\infty} d(y,\omega)
    	\;\leq\; {L} d(y,\omega).\]
It follows that, for each fixed $z\in U(\omega)$,
\begin{eqnarray*}
    	\lim_{\substack{{d(y,\omega)}\to 0\\ y\in U(\omega)}} m_\mu (y,z,\omega) \geq 
    	\lim_{\substack{{d(y,\omega)}\to 0\\ y\in U(\omega)}} 
      	\left(- \frac{\mu}{{L}} \log \big(L d(y,\omega)\big) + c\right) = +\infty.
\end{eqnarray*}
\end{proof}

We analyze further the behavior of the $m_\mu$'s near  $\partial U(\omega)$ and show that  they are  controlled from below by a cone. This fact is used in Proposition \ref{prop::liminf-H} to  investigate the properties of the effective Hamiltonian and is an indispensable estimate to establish the proof of homogenization.
\begin{lemma}[Lipschitz growth near the boundary]\label{lemma::improve-Lip}
Assume $(A2)$ and  $(S2)$. For any $\eta>0$, there exists $\ovl \delta_0=\ovl \delta_0(\eta)>0$ such that, for all $\delta\in (0,\ovl \delta_0)$ and  all $y_1,z\in U^\delta(\omega)$,  $y_2 \in U(\omega)\setminus U^\delta(\omega)$,
\begin{equation}\label{eq:liminf}
    	m_\mu(y_2,z,\omega)-m_\mu(y_1,z,\omega) \geq 
    	-\left(\frac{\mu}{\delta}+{\eta}\right)|y_2-y_1|.
\end{equation}
\end{lemma}

\begin{proof}
Fix $\eta>0$, let $\ovl\delta(\eta)>0$ and $\delta_0>0$  be given by $\cS_\mu(U(\omega))$ and  $(S2)$ and set $\ovl \delta_0 (\eta) := \min(\delta_0,\ovl\delta(\eta))$. In light of Lemma \ref{lemma::Lipschitz}, $m_\mu(\cdot,z,\omega)$ is Lipschitz continuous on $U^\delta(\omega)$ with Lipschitz constant at most $\mu/\delta+{\eta}$. We extend $m_\mu(\cdot,z,\omega)$ 
to the whole space $\R^n$ by
	\[ w_\mu(y,\omega) := \sup_{x\in U^\delta(\omega)}
   	\left\{ m_\mu(x,z,\omega) - \left(\frac{\mu}{\delta}+\eta\right) |y-x|\right\},\]
so that $\|Dw_\mu\|_{\infty}\leq {\mu}/{\delta}+{\eta}.$
It is then immediate  from Proposition \ref{prop::blow-up} that
	\[ m_\mu(y_2,z,\omega) - m_\mu(y_1,z,\omega)\geq 
   	w_\mu(y_2,\omega) - w_\mu(y_1,\omega) \geq -\left(\frac{\mu}{\delta}+{\eta}\right)|y_2-y_1|.\]
\end{proof}

We use the maximality of $m_\mu$ and its behavior near the boundary to show next that $m_\mu$ is a pseudo-metric and thus a subadditive quantity in $U(\omega)$.

\begin{lemma}[Pseudo-metric]\label{lem::mp-sym} 
Assume $(A2)$ and $(S2)$. Then $m_\mu(\cdot,\cdot,\omega):U(\omega)\times U(\omega)\to[0,\infty)$ is symmetric and subadditive, that is, for all $x,y,z\in U(\omega)$,
	\[m_\mu(y,z,\omega)=m_\mu(z,y,\omega) \;\;\;\hbox{ and } \;\;\;
	  m_\mu(x,z,\omega)\leq m_\mu(x,y,\omega) + m_\mu(y,z,\omega).\]
\end{lemma}

\begin{proof} Let $z\in U(\omega)$ and rewrite the formula for the maximal subsolution as
\begin{eqnarray*}
    	m_\mu(y,z,\omega) & = &
      	\sup \left\{ w(y,\omega)-w(z,\omega): w(\cdot,\omega) \in {\cL_\mu}(U(\omega))
      	\hbox{ and } a(y,\omega)|Dw| \le \mu \ \text{in} \ U(\omega) \right \}.
\end{eqnarray*}
For each $w(\cdot,\omega)\in{\cL_\mu}(U(\omega))$, consider its reflection $\widetilde w(y,\omega) := - w(y,\omega)$ and observe that
\begin{eqnarray*}
    	m_\mu(y,z,\omega) & = &
     	 \sup \left\{ \widetilde w(z,\omega)-\widetilde w(y,\omega): \widetilde w(\cdot,\omega) 
	 	\in {\cL_\mu}(U(\omega))
      	\hbox{ and } a(y,\omega)|D\widetilde w| \le \mu \ \text{in} \ U(\omega) \right \}\\ & = & 
      	m_\mu(z,y,\omega).
 \end{eqnarray*}
Fix now $x,z\in U(\omega)$ and recall that $m_\mu(\cdot,z,\omega)$ is the maximal subsolution of (\ref{eq:metric-pb}) in $U(\omega)$ such that $m_\mu(z,z,\omega)=0$. Since the supremum of viscosity subsolutions is still a subsolution, note that
      \[w(\cdot,\omega) = m_\mu(\cdot,x,\omega) - m_\mu(z,x,\omega)\]
is also a subsolution  of (\ref{eq:metric-pb}) in $U(\omega)$, $w(z,\omega)=0$ and, in view of Lemma \ref{lemma::Lipschitz}, $w\in \cS_\mu(U(\omega))$. But $w(\cdot,\omega)$ cannot be extended to a function in ${\cL}$ due to its behavior at $\partial U(\omega)$. To overcome this, we use the following cut-off argument. For each $M>0$,  let
\begin{equation*}
    	w_M(y,\omega) :=
	\begin{cases}
		\min\{M,w(y,\omega)\} &\hbox{ for }\ y\in U(\omega),\\
		M &\hbox{ for }\ y \in \R^n \setminus U(\omega).
	\end{cases}
\end{equation*}
In light of Proposition \ref{prop::blow-up}, for $M$ large enough, there exists $\delta_M\in(0,{\delta_0})$ such that $w_M=M$ in $\R^n \setminus U^{\delta_M}(\omega)$. It follows from Barron and Jensen \cite{BJ:90} and Barles \cite{B:93} that $w_M(\cdot,\omega)$ is a subsolution of (\ref{eq:metric-pb}) in $U(\omega)$, since it is the minimum of two subsolutions associated to a convex Hamiltonian in $U(\omega)$. Furthermore  $w_M\in \cL_\mu(U(\omega))$ and, hence, in view of the maximality of $m_\mu(\cdot,z,\omega)$, for all $y\in U(\omega)$,
	\[w_M(y,\omega) \leq m_\mu(y,z,\omega).\]
Letting $M\to\infty$ yields the claim.
\end{proof}

\begin{remark}\rm
In order to construct admissible test functions $w$ which are globally Lipschitz, we will frequently use the cut-off argument above.
\end{remark}

\subsection*{Well-posedness}

It follows from Perron's method for viscosity solutions (\cite{I:87}) that \eqref{eq:metric-pb} is well-posed for each $\mu>0$ and $m_\mu$ is actually a solution of \eqref{eq:metric-ext-linear}. Since the proof requires some special care  due to the fact that the equation is restricted to $U(\omega)$ and solutions are not globally Lipschitz, we present next some of the details.
\begin{proposition}[Well-posedness]\label{prop::well-posed-mp}
Assume $(A2)$ and $(S2)$. For each $\mu> 0$
and $z\in U(\omega)$, the maximal subsolution $m_\mu(\cdot,z,\omega)$  is a solution of
\begin{equation}\label{eq:metricpb-wp}
  	\begin{cases}
     		 a(y,\omega) |Dm_\mu| = \mu & \hbox{ in }\ U(\omega) \setminus \{z\},\\ 
      		a(y,\omega) |Dm_\mu| \leq \mu & \hbox{ in }\ U(\omega),\\ 
      		m_\mu(\cdot,z,\omega)= 0& \hbox{ at } \{z\}.
 	 \end{cases}
\end{equation}
\end{proposition}

\begin{proof}
In view of Lemma \ref{lemma::Lipschitz}, it is only necessary to show that $m_\mu$ is a supersolution of \eqref{eq:metricpb-wp} in $U(\omega)\setminus \{z\}$. We argue by contradiction. 

Assume $m_\mu(\cdot,z,\omega)$ is not a supersolution in $U(\omega)\setminus\{z\}$. Then there exists  $y_0\in U(\omega)\setminus\{z\}$ and  $\phi(\cdot,\omega)\in\mathcal C^1(U(\omega))$ such that $m_\mu(y_0,z,\omega) = \phi(y_0,\omega)$, $m_\mu(\cdot,z,\omega) >  \phi (\cdot,\omega)$ in $U(\omega)\setminus \{z\}$ and
\begin{equation}\label{eq:Perron}
  	a(y_0,\omega)|D\phi(y_0,\omega)| < \mu.
\end{equation}
For any  $\eps>0$, let $\psi^\eps(y,\omega):=  \phi(y,\omega) + \eps - |y-y_0|^2$ and consider, for all $y\in U(\omega)$,
	\[w^\eps(y,\omega): = \sup\big( m_\mu(y,z,\omega), \psi^\eps(y,\omega)\big).\]
Note that $w^\eps(y,\omega) =\psi^\eps(y,\omega)$ only if $y\in B_r(y_0)$, with $r=\sqrt\eps$.
For $M$ large enough, let 
\begin{equation*}
    	w_M^\eps(y,\omega):=
	\begin{cases}
		\min\{M,w^\eps(y,\omega)\} &\hbox{ for }\ y\in U(\omega),\\
		M &\hbox{ for }\ y \in \R^n \setminus U(\omega).
	\end{cases}
\end{equation*}
We show below that, for $\eps>0$ sufficiently small, $w_M^\eps(\cdot,\omega)\in {\cL_\mu}(U(\omega))$ and, hence, is a subsolution of \eqref{eq:metricpb-wp} in $U(\omega)$. Since  it satisfies $w^\eps_M(y_0,\omega)= \phi(y_0,\omega)+\eps > m_\mu(y_0,z,\omega)$, the claim contradicts the maximality of $m_\mu(\cdot,z,\omega)$ at $y_0$. 

It suffices to check that, for $\eps>0$ sufficiently small, $w^\eps(\cdot,\omega)\in \cS_\mu(U(\omega))$. Since $m_\mu\in \cS_\mu(U(\omega))$,  for any $\eta>0$, let $\ovl\delta(\eta)>0$ such that, for all $\delta\in(0,\ovl\delta)$ and all $y_1,y_2\in U^\delta(\omega)$,
\begin{eqnarray*}
	|m_\mu(y_1,z,\omega)- m_\mu(y_2,z,\omega)|\leq 
  	\left(\frac{\mu}{\delta}+{\eta}\right)|y_1-y_2|.
\end{eqnarray*}
If $y_1,y_2\in U^\delta(\omega)\setminus B_r(y_0)$, then the estimate  above holds for $w^\eps(\cdot,\omega)$, since $w^\eps(\cdot,\omega) = m_\mu(\cdot,z,\omega)$, while if $y_1,y_2\in B_r(y_0)$, the smoothness of $\psi^\eps(\cdot,\omega)$ gives
\begin{eqnarray*}
  	|\psi^\eps(y_1,\omega)-\psi^\eps(y_2,\omega)| \leq  
  	\|D\psi^\eps\|_{L^\infty(B_r(y_0))}|y_1-y_2|.
\end{eqnarray*}
Note also that \eqref{eq:Perron} and the continuity of $a(\cdot,\omega)$, taking if necessary a smaller $\eps$, yield that $\psi^\eps$ is a smooth subsolution of \eqref{eq:metric-pb} in $B_r(y_0)$, hence,
\begin{eqnarray*}
  	|\psi^\eps(y_1,\omega)-\psi^\eps(y_2,\omega)| \leq \frac{\mu}{\inf_{B_r(y_0)}a(\cdot,\omega)}|y_1-y_2|.
\end{eqnarray*}
For $\delta>0$ sufficiently small we have $ B_r(y_0)\cap U^\delta(\omega)\neq\emptyset$. The continuity of $a(\cdot,\omega)$ yields that, for any $\nu>0$, there exists some small $r>0$ such that $\inf_{B_r(y_0)}a(\cdot,\omega)\geq \delta-\nu$. Choosing $\nu=\delta^2\eta/(2\mu+2\delta\eta)$ so that $\nu\to0$ as $\eta\to 0$, we further get 
\begin{eqnarray*}
	 |\psi^\eps(y_1,\omega)-\psi^\eps(y_2,\omega)| \leq \frac{\mu}{\delta-\nu}|y_1-y_2|
	 <\left(\frac\mu\delta+\eta\right)|y_1-y_2|.
\end{eqnarray*}
Lastly if $y_1 \in B_r(y_0)\cap U^\delta(\omega)$  and $y_2\in U^\delta(\omega)\setminus B_r(y_0)$ are such that $w^\eps(y_1,\omega)=\psi^\eps(y_1,\omega)$ and $w^\eps(y_2,\omega)=m_\mu(y_2,z,\omega)$, consider $y_3 \in B_r(y_0)$ on the segment determined by $y_1$ and $y_2$, that is $y_3 = \alpha y_2+(1-\alpha)y_2$ for some $\alpha\in(0,1)$, such that $w^\eps(y_3,\omega)=\psi^\eps(y_3,\omega)=m_\mu(y_3,z,\omega)$. Then, in view of the previous cases, we get
\begin{eqnarray*}
  	|w^\eps(y_1,\omega)-w^\eps(y_2,\omega)| 
   	& \leq & |w^\eps(y_1,\omega)-w^\eps(y_3,\omega)|+ |w^\eps(y_3,\omega)-w^\eps(y_2,\omega)| \\
   	& \leq & \left(\frac \mu \delta+{\eta}\right) |y_1-y_3| + \left(\frac \mu \delta+{\eta}\right) |y_2-y_3|=
      	 \left(\frac{\mu}{\delta}+{\eta}\right)|y_1-y_2|.
\end{eqnarray*}
\end{proof}

We continue with several properties of $m_\mu$ with respect to $\mu$ that are used later in the paper. Since the proof is similar to the one in \cite{AS:13} we omit it here.
\begin{proposition}[Continuity and monotonicity in $\mu$] \label{prop::monotone-mu}
Assume $(A2)$ and $(S2)$. For $\mu_1<\mu_2$, let $m_{\mu_1}$ and $m_{\mu_2}$ be the maximal subsolutions of \eqref{eq:metric-pb}. Then, for any $z\in U(\omega)$,
	\[m_{\mu_1}(\cdot,z,\omega)\leq m_{\mu_2}(\cdot,z,\omega)\hbox{ in }U(\omega).\]
Moreover $\mu \mapsto m_\mu(y,z,\omega)$ is continuous on $(0,\infty)$, locally uniformly in $(y,z)\in U(\omega)\times U(\omega)$, in the sense that, if $\mu_j\to\mu$ as $j\to\infty$, then for every compact $K\subset U(\omega)$,
	\[m_{\mu_j}(\cdot,\cdot,\omega)\to m_\mu(\cdot,\cdot,\omega)\ \hbox{ uniformly in } K\times K.\]
\end{proposition}

\medskip\section{The Homogenization of the Metric Problem}\label{sec::homog-metric-pb}

We use the subadditive ergodic theorem to establish an averaging result for the solution of the metric problem in each unbounded connected component of $\{x\in\R^n : a(x,\omega)\neq0\}$. The argument is not standard, since $m_\mu$ is not defined everywhere in $\R^n$. To overcome this difficulty, we first extend $m_\mu$ along rays starting from the origin,  apply the subadditive ergodic theorem to the extension and then relate back the limit  to the original function. This difficulty is amplified by the blow up of the solution on the boundary of the connected components. To overcome the latter, we construct an averaged metric $\delta$ away from the boundary, then prove that the limit is independent of $\delta$.

To fix the ideas, we assume $U(\omega)$  is an open  unbounded connected component of $\{x\in\R^n : a(x,\omega)>0\}$ with $\partial U(\omega) \subseteq \{x\in\R^n : a(x,\omega) = 0\}$. We  show the homogenized limit exists, it is  subadditive, 1-positively homogeneous and globally Lipschitz. When $U(\omega)$ is an unbounded connected component of $\{x\in\R^n : a(x,\omega)<0\}$, we construct similarly a superadditive averaged  metric.

\begin{theorem}[Averaging of $m_\mu$]\label{thm::avg-metric}
Assume $(A1),(A2),(A3)$, $(S1),(S2)$ and $(S3)$. There exist $\ovl\delta_0>0$ sufficiently small, an event of full probability $\widetilde\Omega\subseteq\Omega$ and, for each $\mu>0$, $\ovl m_\mu:\R^n\to\R$, such that, for every $\omega\in\widetilde\Omega$,  $y,z\in\R^n$ and every $\delta\in(0,\delta_0)$, 
\begin{equation}\label{eq:avg-general}
    	\lim_{t\to\infty} \sup_{|x|\leq R} 
    	\limsup_{\substack{y-z \to x \\ ty,tz\in  U^\delta(\omega)}}
    	\left|\frac1t m_\mu (ty,tz,\omega) - \ovl m_\mu(x)\right|=0.
\end{equation}
\end{theorem}

Since the proof of Theorem \ref{thm::avg-metric} is long, we divide it in several lemmata, which we state and prove first. We begin by establishing the averaging result at the origin $z=0$ for each fixed direction $y$.
\begin{lemma}\label{lemma::avg-metric-orig}
There exists $\ovl \delta_0>0$ such that, for each $\mu>0$, $\delta\in(0,\ovl\delta_0)$ and $y\in\R^n$, there exists a set of full probability $\Omega^{\mu,\delta}_y\subseteq\Omega$ and a Lipschitz extension  $\widetilde m_\mu^\delta(\cdot,\cdot \omega):\R^+ y\times\R^+ y\to\R$ of  $m_\mu(\cdot,\cdot,\omega)$  such that, for every $\omega\in\Omega^{\mu,\delta}_y$,
\begin{equation*}
 \ovl m^\delta_\mu(y,\omega):= \lim_{t\to\infty} \frac1t \widetilde m^\delta_\mu(ty,0,\omega)	.
\end{equation*}
\end{lemma}
\begin{proof} 
We first define an extension of $m_\mu$ in each direction $y\in\R^n$ and then we introduce a new random process which satisfies the assumptions of the subadditive ergodic theorem. 

Fix $y\in\R^n$ with $|y|=1$. For each $t\geq 0$ let $t_*y$ and $t^*y$ be the exit point from $\ovl U^\delta(\omega)$ and the entrance point in $\ovl U^\delta(\omega)$ before and after $ty$ respectively, where
\begin{eqnarray*}
 	t_* :=  \sup\left\{0\leq s \leq t : sy\in \ovl U^\delta(\omega)\right\} \hbox{ and }
  	t^* :=  \inf\left\{s \geq t : sy\in \ovl U^\delta(\omega)\right\}. 
\end{eqnarray*}
It is immediate that $t^*$ and $t_*$ are measurable with respect to $\cF$, $0\leq t_*\leq t\leq t^*$ and there exists $\alpha\in(0,1)$ such that $t = (1-\alpha)\; t_* + \alpha t^*.$ 
Note that if $ty\in \ovl U^\delta(\omega)$, then $ty = t_*y=t^*y$.

We define the extension $\widetilde m^\delta_\mu(\cdot,\cdot,\omega):\R^+ y\times\R^+ y\to\R$ of $m_\mu(\cdot,\cdot,\omega)$ as the bilinear interpolation 
\begin{eqnarray}\nonumber
	\widetilde m^\delta_\mu(ty,sy,\omega) : = &
    	(1-\alpha)\; \big((1-\beta)\ m_\mu(t_*y,s_*y,\omega) + \beta\ m_\mu(t_*y,s^*y,\omega)\big) + \\
    	& \;\;\;\;\; \alpha \
      	\big((1-\beta)\ m_\mu(t^*y,s_*y,\omega) + \beta\ m_\mu(t^*y,s^*y,\omega)\big),
   	 \label{eq:metric-ext-bilinear}  
\end{eqnarray}
where $s:=(1-\beta)s_*+\beta s^*$ with $\beta \in (0,1)$. In particular, if $0\in \ovl U^\delta(\omega)$, $\widetilde m^\delta_\mu(\cdot,0,\omega):\R^+y\to\R$ is given by
\begin{equation}\label{eq:metric-ext-linear}
    	\widetilde m^\delta_\mu(ty,0,\omega) = 
    	(1-\alpha)\ m_\mu(t_*y,0,\omega) + \alpha\ m_\mu(t^*y,0,\omega).
\end{equation}
In view of  $(S3)$, $\widetilde m^\delta_\mu(\cdot,\cdot,\omega)$ is well defined. Moreover, if   $\ovl U_y^\delta(\omega)= \ovl U^\delta(\omega)\cap \R^+ y$, then
	\[  \widetilde m_\mu^\delta(\cdot,\cdot,\omega)  = m_\mu(\cdot,\cdot,\omega) 
   	 \; \hbox{ on } \;\;\; \;  \ovl U_y^\delta(\omega) \times  \ovl U_y^\delta(\omega).\]

Given $\widetilde m^\delta_\mu(\cdot,\cdot,\omega)$, we define the random process $Q:\mathcal I\to L^1(\Omega,\bP)$ by 
\[Q([s,t))(\omega):= \widetilde m^\delta_\mu(ty,sy,\omega),\]
which is a continuous subadditive process on $(\Omega,\cF,\bP)$ endowed with $\big(\tau_{ty}\big)_{t\geq 0}$. 

Without any loss of generality, it is enough to check the claim for the linear interpolation \eqref{eq:metric-ext-linear}, since the bilinear interpolation  \eqref{eq:metric-ext-bilinear} is linear in each of its arguments. We assume hence that $0\in \ovl U^\delta(\omega)$,  otherwise we  consider $0^*y$ the closest point to the origin in direction $y$ and work directly with the bilinear interpolation \eqref{eq:metric-ext-bilinear}.

The  $\widetilde m^\delta_\mu$'s  preserve the translation invariance of $m_\mu$, since
\begin{eqnarray*}
   	 \widetilde m^\delta_\mu(ty,0,\tau_{sy}\omega) 
    	& = & (1-\alpha)\ m_\mu(t_*y,0,\tau_{sy}\omega) + \alpha\ m_\mu(t^*y,0,\tau_{sy}\omega)\\
    	& = & (1-\alpha)\ m_\mu(sy + t_*y,sy,\omega) + \alpha\ m_\mu(sy+t^*y,sy,\omega)\\
    	& = & (1-\alpha)\ m_\mu((s+t)_*y,sy,\omega) + \alpha\ m_\mu((s+t)^*y,sy,\omega)\\
    	& = & \widetilde m^\delta_\mu(((s+t)y,sy,\omega);
\end{eqnarray*}
note that we used the translation invariance of the level sets of $a(\cdot,\omega)$ to  say that $sy\in \ovl U^\delta(\omega)$ if and only if $0\in \ovl U^\delta(\tau_{sy}\omega)$ and 
\begin{eqnarray*}
    	(s+t)_* (\omega)
    	& = & \sup\left\{s\leq r \leq s+t : ry\in \ovl U^\delta(\omega)\right\}\\
    	& = & \sup\left\{0\leq r \leq t : ry+sy\in \ovl U^\delta(\omega)\right\}+s\\
    	& = & \sup\left\{0\leq r \leq t : ry\in \ovl U^\delta(\tau_{sy}\omega)\right\}+s = t_*(\tau_{sy}\omega)+s.
\end{eqnarray*}
The $\widetilde m^\delta_\mu$'s also preserve the Lipschitz constants of $m_\mu$'s on $\ovl U_y^\delta(\omega)$. Indeed, for all $t>0$, 
\begin{eqnarray*}
    	\widetilde m^\delta_\mu(ty,0,\omega)
    	& = & (1-\alpha)\;m_\mu(t_*y,0,\omega) + \alpha\; m_\mu(t^*y,0,\omega)\\
    	& \leq &  (1-\alpha)\; \left( \frac\mu\delta + \eta \right)\; t_* |y| +  
               \alpha \; \left( \frac\mu\delta + \eta \right)\; t^* |y|
    	=\; \left( \frac\mu\delta + \eta \right)\; t|y|.
\end{eqnarray*}
Finally the $\widetilde m^\delta_\mu$'s remain subadditive. Indeed, let $t>s>0$ and assume again, without loss of generality, that $sy\in \ovl U^\delta(\omega)$. Then
\begin{eqnarray*}
    	\widetilde m^\delta_\mu(ty,0,\omega) 
    	&  =  &  (1-\alpha)\ m_\mu(t_*y,0,\omega) + \alpha\  m_\mu(t^*y,0,\omega) \\
    	&\leq &  (1-\alpha) \big(m_\mu(t_*y,sy,\omega) + m_\mu(sy,0,\omega) \big) +
		\alpha  \big( m_\mu(t^*y,sy,\omega) + m_\mu(sy,0,\omega)\big) \\
    	&  =  &  \widetilde m^\delta_\mu(ty,sy,\omega) + \widetilde m^\delta_\mu(sy,0,\omega). 
\end{eqnarray*}
All the above properties remain true for the bilinear extension $\widetilde m^\delta_\mu=\widetilde m^\delta_\mu(ty,sy,\omega)$. Then, by the subadditive ergodic theorem, there exists an event $\Omega^{\mu,\delta}_y$ of full probability such that, for all $\omega\in {\Omega^{\mu,\delta}_y}$, there exists 
\begin{equation*}
    	\ovl m^\delta_\mu(y,\omega): =
    	\lim_{t\to\infty} \frac 1t \widetilde m^\delta_\mu(ty,0,\omega)\;=
    	\lim_{\substack{t\to\infty\\ ty,0\in \ovl U^\delta(\omega)}} \frac 1t m_\mu(ty,0,\omega). 
\end{equation*}
\end{proof}

\begin{remark}\label{rk:0inU}\rm
Note that we may assume, without any loss of generality, that  $0\in U(\omega)$ a.s., which yields that $0 \in \ovl U^\delta(\omega)$ for $\delta>0$ sufficiently small. Otherwise, for any unit vector $y\in\R^n$ and $\delta>0$, we shift the origin in direction $y\in\R^n$ to the closest point which lies inside the domain $\ovl U^\delta(\omega)$, namely we replace $0$ by $0^*y$, where
 \[0^* =  \inf\{t\geq 0: ty\in \ovl U^\delta(\omega)\}.\] 
In this case, instead of arguing for the linear interpolation \eqref{eq:metric-ext-linear}, we work directly with the bilinear interpolation \eqref{eq:metric-ext-bilinear}.
\end{remark}

We show next that $\ovl m_\mu^\delta$ is deterministic and establish the a.s. convergence.
\begin{lemma}\label{lemma::avg-metric-orig-det}
Let $\eta>0$ and $\ovl\delta_0=\ovl\delta_0(\eta)>0$ be given by Lemma \ref{lemma::Lipschitz}. For each $\mu>0$ and $\delta\in(0,\ovl\delta_0)$, there exists a set of full probability $\Omega^{\mu,\delta}\in\cF$ and  $\ovl m^\delta_\mu:\R^n\to\R$ such that,  for every $\omega\in \Omega^{\mu,\delta}$ and $y\in\R^n$, 
\begin{equation}\label{eq:avg-origin}
  	\ovl m^\delta_\mu(y) := \lim_{\substack{t\to\infty\\ ty\in \ovl U^\delta(\omega)}}
    	\frac1t m_\mu (ty,0^*y,\omega).
\end{equation}
For fixed $\mu>0$ and $\delta\in(0,\ovl\delta_0)$, $\ovl m^\delta_\mu$ is subadditive, $1$-positively homogeneous and Lipschitz continuous with Lipschitz constant at most $\mu/\delta+\eta$.
\end{lemma}

\begin{proof} 
In light of Remark \ref{rk:0inU}, we may assume that $0\in U(\omega)$ a.s..  It remains to show that $\ovl m_\mu^\delta(y,\cdot)$ is deterministic. 
In view of the ergodicity this would follow once we show that $\ovl m^\delta_\mu$ is translation invariant, that is, it satisfies for all $z\in\R^n$,
\begin{equation}\label{eq:stationarity-avg-m}
  	\ovl m^\delta_\mu(y,\omega) = \ovl m^\delta_\mu(y,\tau_z\omega).
\end{equation}
We establish \eqref{eq:stationarity-avg-m} first for $z\in \ovl U^\delta(\omega)$ and then deduce the general case.

Let $z\in \ovl U^\delta(\omega)$. In view of $(S3)$, for each $\delta\in(0,\ovl\delta_0)$ and $y\in\R^n\setminus \{0\}$, there exist $\big(l_j(\omega)y\big)_{j\geq 0}\subset \ovl U^\delta_y(\omega)$ and $\big(r_j(\tau_z\omega)y+z\big)_{j\geq 0}\subset \ovl U^\delta_y(\omega)$ such that, as $j\to\infty$,  $l_j(\omega)\to\infty$, $r_j(\tau_z \omega)\to\infty$ and {$l_j(\omega)/r_j(\tau_z \omega)\to 1$}. Then, by the subadditivity and Lipschitz continuity of $m_\mu$,
 \begin{eqnarray*}
  	\ovl m^\delta_\mu(y,\tau_z\omega) 
   	& =   & 
      	\lim_{j\to\infty}\frac{1}{r_j(\tau_z \omega)} m_\mu(r_j(\tau_z \omega)y, 0, \tau_z\omega) =
      	\lim_{j\to\infty}\frac{1}{r_j(\tau_z \omega)} m_\mu(r_j(\tau_z \omega)y + z , z, \omega)\\
   	& \leq & 
      	\lim_{t\to\infty}\frac{1}{r_j(\tau_z \omega)} \left(m_\mu(r_j(\tau_z \omega)y+z,l_j(\omega)y,\omega) + 
	m_\mu(l_j(\omega)y,0,\omega) +  m_\mu(0,z,\omega) \right)\\
   	& \leq & 
      	\lim_{j\to\infty}\frac{1}{r_j(\tau_z \omega)} \left(m_\mu(l_j(\omega)y,0,\omega) + \left( \frac\mu\delta + \eta \right) 
	 |r_j(\tau_z \omega)-l_j(\omega)| |y|+2\left( \frac\mu\delta + \eta \right)  |z| \right)\\
  	& =  &
     	\lim_{j\to\infty}\frac{1}{l_j(\omega)} m_\mu(l_j(\omega)y,0,\omega)= \ovl m^\delta_\mu(y,\omega). 
 \end{eqnarray*}
A similar argument gives the reverse inequality and, hence, for all $z\in \ovl U^\delta(\omega)$, \eqref{eq:stationarity-avg-m} holds. 

Let $z'\in\R^n$. Then
\begin{eqnarray*}
    	\ovl m^\delta_\mu(y,\tau_{z'}\omega) 
    	=  \sup_{z\in \ovl U^\delta(\tau_{z'}\omega)}\ovl m^\delta_\mu(y,\tau_z\tau_{z'}\omega)
    	=  \sup_{z+z'\in \ovl U^\delta(\omega)}\ovl m^\delta_\mu(y,\tau_{z+z'}\omega)
    	=   \ovl m^\delta_\mu(y,\omega). 
\end{eqnarray*}

It is clear, from the construction, that the average metric $\ovl m^\delta_\mu$ is 1-positively homogeneous.

We show next that, in view of Lemma \ref{lemma::avg-metric-orig},  $\ovl m^\delta_\mu$ is Lipschitz continuous. Indeed, let $y,y'\in\R^n \setminus\{0\}$ and $\omega\in{\Omega^{\mu,\delta}_y}\cap \Omega^{\mu,\delta}_{y'}$. In view of $(S3)$, there exist two sequences $\big(l_j(\omega)y'\big)_{j\geq 0},\big(r_j(\omega)y\big)_{j\geq 0} \subset \ovl U^\delta(\omega)$ such that, as $j\to\infty$, $l_j(\omega)\to\infty, r_j(\omega)\to\infty$,  and  {$l_j(\omega)/r_j(\omega)\to 1$}. Then the subadditivity and Lipschitz continuity of $m_\mu$ yield the Lipschitz continuity of $\ovl m_\mu^\delta$ as follows:
\begin{eqnarray*}
    	\ovl m^\delta_\mu(y) 
    	& = & 	\lim_{j\to\infty} \frac{1}{r_j(\omega)} m_\mu(r_j(\omega)y,0,\omega) \\
    	& \leq & 	\lim_{j\to\infty}\frac{1}{r_j(\omega)} \big(m_\mu(l_j(\omega)y',0,\omega) 
      			+ m_\mu(r_j(\omega)y,l_j(\omega)y',\omega)\big)\\
    	& \leq & 	 \lim_{j\to\infty}\frac{1}{r_j(\omega)} \left( m_\mu(l_j(\omega)y',0,\omega)+
     			 \left(\frac\mu\delta+\eta\right)\left(|r_j(\omega)||y-y'|+|r_j(\omega)-l_j(\omega)||y'|\right)\right) \\  
    	& \leq &	 \ovl m_\mu^\delta(y')+\left( \frac\mu\delta + \eta \right)|y-y'|.
\end{eqnarray*}

The subadditivity of $\ovl m_\mu^\delta$ follows directly from the subadditivity of $m_\mu(\cdot,\cdot,\omega)$ in $U(\omega)$. For $\big(l_j(\omega)y\big)_{j\geq 0},\big(r_j(\omega)(y+z)\big)_{j\geq 0} \subset \ovl U^\delta_y(\omega)$  as before, we have 
\begin{eqnarray*}
	 \ovl m_\mu(y+z) 
	 &   =   & \lim_{j \to \infty} \frac{1}{r_j(\omega)} m_\mu(r_j(\omega)(y+z),0,\omega) \\
 	 & \leq & \lim_{j \to \infty} \frac{1}{r_j(\omega)} \big( m_\mu(r_j(\omega)(y+z),l_j(\omega) y,\omega) + 
	 m_\mu(l_j(\omega) y,0,\omega)\big) =  \ovl m_\mu(y) +  \ovl m_\mu(z) .
\end{eqnarray*}

That the average holds for an event of full probability for all $y\in\R^n$ is an immediate consequence of the the density of $\Q^n$, the Lipschitz continuity of $m_\mu$ and assumption $(S3)$.
\end{proof}

Having established Lemma \ref{lemma::avg-metric-orig} and Lemma \ref{lemma::avg-metric-orig-det}, the proof of Theorem \ref{thm::avg-metric} is a consequence of Egoroff's theorem, the subadditive ergodic theorem and the Lipschitz estimates of $m_\mu$. Although the argument has already appeared in several references \cite{LS:10, AS:12}, for the benefit of the reader we present some of the details, since averaging takes place only locally in $U(\omega)$ and not in $\R^n$.

\begin{proof}[Proof of Theorem \ref{thm::avg-metric}]

To simplify the arguments we drop the dependence on $\mu$ of the probability events and assume, in view of Remark \ref{rk:0inU}, that $0\in U(\omega)$ a.s.. For each $R>0$, consider the process 
	\[  \mathcal M([s,t))(\omega) : = \sup_{|x|\leq R}
    	\left|\widetilde m^{\delta,x}_\mu(tx,sx,\omega) - (t-s)\ovl m^\delta_\mu(x)\right|,\]
where $\widetilde m^{\delta,x}_\mu(\cdot,\cdot,\omega) $ is the bilinear extension in direction $x$ of $m_\mu(\cdot,\cdot,\omega)$, given by \eqref{eq:metric-ext-bilinear}. Arguing as in the proof of Lemma \ref{lemma::avg-metric-orig}, it is easy to see that $\mathcal M$ is a continuous subadditive process. The subadditive ergodic theorem then yields, a.s. in $\omega\in\Omega^{\delta}$,
\begin{equation*}
    	\lim_{t\to\infty} \sup_{\substack{|x|\leq R}} 
    	\left| \frac{1}{t} \widetilde m^{\delta,x}_\mu(tx,0,\omega) -  \ovl m^\delta_\mu(x)\right|= 
    	\lim_{t\to\infty} \sup_{\substack{|x|\leq R\\ tx\in \ovl U^\delta(\omega)}} 
    	\left| \frac{1}{t}  m_\mu(tx,0,\omega) -  \ovl m^\delta_\mu(x)\right| = 0.
\end{equation*}
It follows from Egoroff theorem that, for any $\eps>0$, there exists $t_\eps>0$ and an event $W^{\delta,\eps}\subset\Omega^{\delta}$ such that $\bP[\Omega^{\delta}\setminus W^{\delta,\eps}]<\eps^n/4$ and, for all $t\geq t_\eps$, 
\begin{equation}\label{eq:egoroff}
  	\esssup_{\omega\in W^{\delta,\eps}} \sup_{\substack{|x|\leq R \\ tx\in \ovl U^\delta(\omega)}}
  	\left|\frac 1t \widetilde m^\delta_\mu(tx,0,\omega) - \ovl m_\mu^\delta(x) \right| \leq \eps.
\end{equation}
Applying the ergodic theorem to ${\bf 1}_{W^{\delta,\eps}}$ we find an event of full probability $\Omega^{\delta,\eps}\subseteq \Omega^{\delta}$, such that, for all $\omega\in\Omega^{\delta,\eps}$,
\begin{equation*}
  	\lim_{r\to\infty}\frac{1}{|B_r|} \int_{B_r} {\bf 1}_{W^{\delta,\eps}}(\tau_z\omega)dz = 
  	\bP[W^{\delta,\eps}]\geq 1-\frac{\eps^n}{4}.
\end{equation*}
Consider the event 
$\tilde \Omega^\delta  = \bigcap_{\eps\in(0,1)\cap \Q}\Omega^{\delta,\eps}$
such that $\bP[\tilde \Omega^{\delta}] = 1$. Then, for each $\eps>0$ and $\omega\in\Omega^{\delta}$, there exists $r_\eps>0$ such that, for all $r>r_\eps$,
\begin{equation}\label{eq:ergodic}
  	\left|\left\{ z\in B_r : \tau_z\omega\in W^{\delta,\eps}\right\}\right| > \left(1-\frac{\eps^n}{2}\right) |B_r|,
\end{equation}
which implies that no ball of radius $r\eps$ is contained in $\{z\in B_r: \tau_z\omega\not\in W^{\delta,\eps}\}$. 

In view of $(S3)$, for any $z\in B_r\setminus\{0\}$, there exist a sequence $\big(t_j(\omega)z\big)_{j\geq 0} \subset \ovl U^\delta(\omega)$ such that, as $j\to\infty$, $t_j(\omega)\to\infty$.
We deduce from \eqref{eq:ergodic} that, for each $j$ sufficiently large, there exists $z_j \in B_r$, such that $\tau_{t_j(\omega) z_j}\omega\in W^{\delta,\eps}$ and $|z- z_j|<\eps$.
Then \eqref{eq:egoroff} yields
 \begin{eqnarray*}
     	&&\mathcal M([0,t_j(\omega))) (\tau_{t_j(\omega)z}\omega) = 
       	\sup_{\substack{|x|\leq R \\ t_j(\omega)x \in \ovl U^\delta(\tau_{t_j(\omega)z}\omega)}}
       	\left| m_\mu(t_j(\omega)x, 0,\tau_{t_j(\omega) z}\omega) - t_j(\omega) \ovl m_\mu^\delta(x) \right|\\
    	&& \leq   
       	\sup_{\substack{|x|\leq R \\ t_j(\omega)(x +z_j)\in \ovl U^\delta(\omega)}}
        	\left|m_\mu(t_j(\omega)(x + z_j), t_j(\omega)z_j,\omega) - t_j(\omega) \ovl m_\mu^\delta(x) \right|  + 
        	2t_j(\omega){ \left( \frac\mu\delta + \eta \right) } |z- z_j|\\
     	&& = 
       	\sup_{\substack{|x|\leq R \\ t_j(\omega)x\in \ovl U^\delta(\tau_{t_j(\omega)z_j}\omega)}}
       	\left|m_\mu(t_j(\omega)x,0,\tau_{t_j(\omega)z_j}\omega) - t_j(\omega)\ovl m_\mu^\delta(x) \right| + 		2t_j(\omega){\left(\frac{\mu}{\delta}+\eta\right)}\eps\\
       	&&\leq t_j(\omega) \eps + 2t_j(\omega){\left(\frac{\mu}{\delta}+\eta\right)}\eps.
 \end{eqnarray*} 

Letting $\eps\to 0$ we find
	\[ \lim_{j\to\infty} \frac{1}{t_j(\omega)}\mathcal M([0,t_j(\omega)))(\tau_{t_j(\omega)z}\omega) =
	\lim_{t\to\infty}\sup_{\substack{|y-z|\leq R \\ ty,tz\in \ovl U^\delta(\omega)}}
 	\left|\frac 1t m_\mu(ty, tz,\omega) - \ovl m_\mu^\delta(y-z) \right| = 0.\]

Similarly we obtain an event of full probability independent of $\mu$.\smallskip

Finally, we check that $\ovl m_\mu^\delta$ is independent of $\delta$. Let $0<\delta_2<\delta_1<\ovl{\delta_0}$ such that $U^{\delta_i}(\omega)\subset U(\omega)$ for $i=1,2$, which, by (S2), remain connected.  In view of $(S3)$, for each fixed $y\in\R^n$ there exists a sequence $(t_j(\omega))_{j\geq 0}$ such that $(t_j(\omega) y)_{j\geq 0}\subset U^{\delta_2}(\omega)\subset U^{\delta_1}(\omega)$ and as $j\to\infty$, $t_j(\omega)\to\infty$. Consider  the event of full probability
$\ovl\Omega = \bigcap_{n\geq 1}\Omega^{\frac1n}.$ Then, for every $\omega \in \ovl\Omega$ ,
\begin{eqnarray*}
    	\ovl m^{\delta_1}_\mu(y) =  \lim_{j\to\infty} \frac{1}{t_j(\omega)} m_\mu(t_j(\omega)y,0,\omega) 
	= \ovl m^{\delta_2}_\mu(y).
\end{eqnarray*}
\end{proof}

We show in the following lemma that the liminf of the averaged $m_\mu$'s holds all the way up to the boundary.  However, the result  does not hold for the limsup, in view of the loss of Lipschitz estimates of $m_\mu's$ near $\partial U(\omega)$. In the case when $U(\omega)$ is a connected component of $\{x\in\R^n:a(x,\omega)<0\}$ the arguments are reverted: the metric $m_\mu$ is negative and the limsup of the averaged $m_\mu$ holds up to the boundary.

\begin{lemma}\label{lemma::up-to-bdry}
Assume that $(A1),(A2),(A3)$, $(S1),(S2)$ and $(S3)$ hold and let $\cU$ be given by $(S1)$.  For $\mu> 0$, let $m_\mu(\cdot,\cdot,\omega)$ be the solution of the metric problem \eqref{eq:metric-pb} and $\ovl m_\mu$ be the averaged metric corresponding to $\cU$.
For $y,z \in\R^n\setminus \{0\}$, $\omega \in \widetilde\Omega$ and $\delta>0$ sufficiently small,
	\[ \liminf_{\substack{t \to \infty\\ty,tz\in U(\omega)}} \frac{m_\mu(ty,tz,\omega)}{t} = 
	\lim_{\substack{t \to \infty\\ty,tz\in \ovl U^\delta(\omega)}} \frac{m_\mu(ty,tz,\omega)}{t}=\ovl{m}_\mu(y-z).\]
\end{lemma}

\begin{proof}
In view of Remark \ref{rk:0inU} and the Lipschitz continuity of $m_\mu$, we may assume for simplicity that $z=0$. Let $\eta>0$, $\ovl\delta_0(\eta)$ be given by Lemma \ref{lemma::Lipschitz} and fix $\delta\in(0,\ovl\delta_0)$. We only need to check the result for points $ty \in U(\omega) \setminus U^\delta(\omega)$.

In view of $(S3)$, there exist $\left(l_j(\omega) y\right)_{j\geq 0}\subset U(\omega) \setminus U^\delta(\omega)$  and  $\left(r_j(\omega) y\right)_{j\geq 0} \subset U^\delta(\omega)$ such that, as $j\to\infty$,  $l_j(\omega) \to \infty$, $r_j(\omega) \to \infty$ and $l_j(\omega)/r_j(\omega)\to 1$. In light of Lemma \ref{lemma::improve-Lip} we have 
\begin{eqnarray*}
    m_\mu(l_j(\omega) y,0,\omega) & \geq & m_\mu(r_j(\omega)y,0,\omega) -\left(\frac{\mu}{\delta}+\eta\right)|l_j(\omega)-r_j(\omega)||y|.
\end{eqnarray*}
Theorem \ref{thm::avg-metric}  yields further that
\begin{eqnarray*}
    \liminf_{j\to \infty} \frac{m_\mu(l_j(\omega) y,0,\omega)}{l_j(\omega)} 
    & \geq &  \lim_{j\to \infty} \left(\frac{m_\mu(r_j(\omega) y,0,\omega)}{l_j(\omega)}- \left(\frac{\mu+\delta\eta}{\delta}\right)\frac{|l_j(\omega)-r_j(\omega)|}{l_j(\omega)}|y| \right)\\
    &   =  &  \lim_{j\to \infty}\frac{m_\mu(r_j(\omega) y,0,\omega)}{r_j(\omega)}\frac{r_j(\omega)}{l_j(\omega)} 	=  \ovl{m}_\mu(y).
\end{eqnarray*}
The conclusion follows.
\end{proof}

The next properties follow  immediately from Lemma \ref{lemma::avg-metric-orig-det} and Proposition \ref{prop::monotone-mu}.
\begin{proposition}[Properties of the averaging metric]
Assume that $(A1),(A2),(A3),(S1),(S2)$ and $(S3)$ hold and let $\cU\in\cF$ given by $(S1)$. For every $\mu>0$, let $\ovl m_\mu :\R^n\to\R$ be  the averaged metric corresponding to $\cU$. Then
\begin{enumerate}[(i)]
    \item $  y\mapsto \ovl m_\mu(y)$ is Lipschitz continuous with Lipschitz constant depending only on $\mu$.
    \item $  y\mapsto \ovl m_\mu(y)$ is subadditive, positively homogeneous and, hence, convex.
    \item $\mu\mapsto \ovl m_\mu(y)$ is increasing and locally uniformly  continuous in $y\in\R^n$.
\end{enumerate}
\end{proposition}

\medskip\section{The Effective Hamiltonian}\label{sec::EffHamiltonian}

Since the connected components $(U_i(\omega))_{i\in I_0}$ are generated by a family $(\cU_i)_{i\in I_0}$ of the probability space $\Omega$, it is possible to show that, for each $U_i(\omega)$ with $i\in I$, there exists a deterministic, effective Hamiltonian $\ovl H_i$.  

As usual, we fix  one connected component $ U(\omega)$ of $\{x\in\R^n: a(x,\omega)>0\}$ and characterize the corresponding effective Hamiltonian $\ovl H(p)$.  Similar arguments remain true for the connected components of $\{x\in\R^n: a(x,\omega)<0\}$.

\subsection*{The characterization of the effective Hamiltonian}
We first define the effective Hamiltonian as the smallest constant for which the metric problem has a subsolution with stricly sublinear decay at infinity. We  then give an inf-sup representation formula for the effective Hamiltonian,  which by the stationarity assumption holds a.s. in $\omega$. Unlike previous results \cite{AS:13, LS:05} where subsolutions are  in $\R^n$, now they are restricted to satisfy the equation in the connected component $U(\omega)$. Moreover, in view of the well posedness of the metric problem, they must satisfy the  constraint imposed by the class $\cL_\mu(U(\omega))$.

The {effective Hamiltonian} $\ovl H(p,\omega)$ corresponding to  $U(\omega)$ is given by
\begin{eqnarray}\label{eq:effH} \nonumber 
  	\ovl H(p,\omega)  =   \inf \Big\{  \mu>0 :  
	    	\hbox{ there exists } w(\cdot,\omega)\in \cS^+
            	\hbox{ with } w(\cdot,\omega)+p\cdot y\in \cL_\mu(U(\omega)) \;\;\; \\
   	    	\hbox{ such that } a(y,\omega)|p+Dw| \leq \mu \hbox{ in } \;\;\;U(\omega)\Big\}.
\end{eqnarray}

In view of Lemma \ref{lem:equiv-soln} the differential inequality in \eqref{eq:effH} may be interpreted either in the viscosity sense or in the almost everywhere sense. Note that $w(\cdot,\omega)$ is defined in $\R^n$, but it is only required to be subsolution in the domain $U(\omega)$.

It follows from the stationarity of $a$ and of the level sets $U(\omega)$ and the ergodicity assumption that the effective Hamiltonian is deterministic. More precisely, the following holds.
\begin{proposition}\label{prop::min-max-H}
Assume $(A1)$ and $(S1)$. There exists a set of full probability $\tilde \Omega\subseteq \Omega$ such that, for every $\omega\in\tilde \Omega$ and $p\in\R^n$, 
\begin{eqnarray*}
  	\ovl H(p) = \ovl H(p,\omega) = 
    	\inf_{\substack{(w,\lambda) \in \cS^+\times(0,\infty)\\
		    w+p\cdot y\in{\cL_\lambda(U(\omega))}}} 
    	\left[\esssup_{y \in U(\omega)} \
    	\Big( a(y,\omega)|p+Dw| \Big)\right].
\end{eqnarray*}
\end{proposition}

In view of the inf-sup representation formula, we establish some immediate properties of $\overline H$, which we state without a proof.
\begin{proposition}[Convexity, Homogeneity]\label{prop::convexity-H}
Assume $(A1)$, $(S1)$ hold and let  $\cU\in\cF$ be given by $(S1)$. The effective Hamiltonian $\ovl H$ corresponding to $\cU$ is convex, 1-positively homogeneous and
\[\min_{p\in\R^n} \ovl H(p)=\ovl H(0)=0.\]
\end{proposition}

The next result is essential to our analysis and establishes that the effective Hamiltonian corresponding to bounded connected components is null.
\begin{proposition}\label{prop::barH-domainU}
Assume $(A1)$, $(S1)$ hold and let $\cU\in\cF$ be given by $(S1)$. If $ \;\ovl H$ is the effective Hamiltonian corresponding to $\cU$ and $U(\omega)$ is bounded, then $\ovl H(p) = 0$.
\end{proposition}

\begin{proof} 
Since $U(\omega)$ is bounded, there exists $R=R(\omega)>0$ such that $U(\omega) \subset B_R$. Let $\phi(\cdot,\omega) \in \mathcal C_c^\infty(\R^n)$ such that
	$ \phi(y,\omega) = -p\cdot y \hbox{ in } \;\;\; U(\omega) \hbox{ and }
   	\phi(\cdot,\omega) \equiv 0\hbox{ in }\R^n \setminus B_{R+1}.$
Then $\phi(\cdot,\omega)\in \cS^+$,  for all $\mu>0$, $\phi(\cdot,\omega)+p\cdot y \in \cL_\mu (U(\omega)))$ and, in view of Proposition \ref{prop::min-max-H},
	\[0 \le \ovl H(p) \le \esssup_{y \in U(\omega)} a(y,\omega)|p+D\phi| =0.\]
\end{proof}

\subsection*{The connection between the effective Hamiltonian and the metric problem}

We establish here the connection between the effective Hamiltonian $\ovl H(p)$ and the solution $m_\mu(y,z,\omega) - p\cdot (y-z)$ of the corresponding metric problem with $\mu = \ovl H(p)$. It turns out that the effective Hamiltonian  is the smallest constant $\mu$ for which $m_\mu(y,z,\omega)-p\cdot (y-z)$ has strictly sublinear decay at infinity, from below. We illustrate in Figure \ref{fig::blow-up-asym} the profile of the maximal solution $m_\mu$, which, although blows up at $\partial U_i(\omega)$, has a sublinear decay from below. Lemma \ref{lemma::improve-Lip} is essential in establishing this result.
\begin{figure}[!ht]
    	\centering\includegraphics[width=0.6\linewidth]{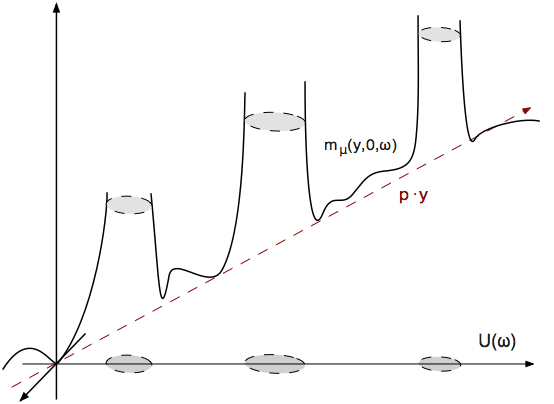} 
    	\caption{Profile of the maximal solution $m_\mu(\cdot,z,\omega)$. Asymptotic strictly sublinear decay.}
    	\label{fig::blow-up-asym}  
\end{figure}

\begin{proposition}\label{prop::liminf-H}
Assume that $(A1)$ and $(S1)$ hold and let $\cU$ be given by $(S1)$. For $\mu> 0$ and $z\in U(\omega)$, let $m_\mu(\cdot,z,\omega)$ and $\ovl H$ be  respectively the maximal subsolution of the metric problem \eqref{eq:metric-pb} and the effective Hamiltonian corresponding to $\cU$. Then, for each $p\in\R^n$ and $\omega \in \ovl\Omega$,
\begin{equation}\label{eq:Hbar-m_sublinear}
   	 \mu\geq \ovl H(p) \ \hbox{ if and only if } \ 
   	 \liminf_{\substack{|y|\to\infty\\y\in U(\omega)}}
    	\frac{m_\mu(y,z,\omega)-p\cdot (y-z)}{|y|} \ge 0.
\end{equation}
\end{proposition}

\begin{proof} 
Fix $\mu >\ovl H(p)$ and recall that there exists $w_\mu(\cdot,\omega)\in \cS^+$ such that  $w_\mu(\cdot,\omega)+p\cdot y \in\cL_\mu(U(\omega))$ and
$w_\mu(y,\omega) - w_\mu(z,\omega)  + p\cdot (y-z)$ is a subsolution of the metric problem \eqref{eq:metric-pb}. The maximality of $m_\mu(\cdot,z,\omega)$ yields that, for all $y\in U(\omega)$,
	\[ m_{\mu}(y,z,\omega) \geq w_\mu(y,\omega) - w_\mu(z,\omega)  + p\cdot (y-z). \] 
From the definition of $\cS^+$ we deduce the inequality in \eqref{eq:Hbar-m_sublinear} holds for $\mu>\ovl H(p)$ and, in view of the continuity of $m_\mu$ with respect to $\mu$, it holds for any $\mu\geq \ovl H(p)$.

Conversely assume the inequality in \eqref{eq:Hbar-m_sublinear} holds for some $\mu> 0$. For $\eta>0$ and $\delta\in(0,\ovl\delta_0(\eta))$ let  $w^\delta_\mu(\cdot,\omega) \in \cS_\mu(U(\omega))$ so that $w^\delta_\mu(\cdot,\omega)=m_\mu(\cdot,z,\omega)$ in $U^\delta(\omega)$ and $\|Dw^\delta\|_\infty\leq (\mu/\delta+\eta)$.
Then $w_\mu(y,\omega): =w^\delta_\mu(y,\omega)-p\cdot (y-z)$ solves 
      \[a(y,\omega)|p+Dw_\mu| \leq \mu \;\;\;\hbox{ in } \;\;\;U(\omega).\]
In addition, for $y\in U(\omega)$, there exists $\tilde y\in U^\delta(\omega)$ such that $|y-\tilde y|\leq C\delta$ and
     \[w_\mu(y,\omega) \geq w^\delta_\mu(\tilde y,\omega)-p\cdot (\tilde y-z) - \left(\frac\mu\delta+\eta\right) |y-\tilde y|.\]
The fact that $w_\mu\in \cS^+$ follows from  Lemma \ref{lemma::up-to-bdry}. Indeed,
    \[ \liminf_{|y|\to\infty} \frac{w_\mu(y,\omega)}{|y|} \ge
    \liminf_{\substack{|y|\to\infty\\y\in U^\delta(\omega)}} 
    \frac{w^\delta_\mu(y,\omega)-p\cdot (y-z)}{|y|} =   
    \liminf_{\substack{|y|\to\infty\\y\in U(\omega)}} 
    \frac{m_\mu(y,z,\omega)-p\cdot (y-z)}{|y|} \ge 0.\]
Thus $w_\mu$  is an admissible function in the definition of $\ovl H(p)$ and, hence, $\mu\geq \ovl H(p)$.
\end{proof}

We next use this result to give a dual formulation for $\ovl H(p)$ and the averaged metric $\ovl m_\mu$ corresponding to $\cU\in\cF$.  Since the proof is similar to the one in \cite{AS:13}, we omit it.
\begin{corollary}[The effective Hamiltonian and the averaged metric]\label{corr::avg-metric-H}
Assume that $(A1),(A2),(A3)$, $(S1)$, $(S2)$ and $(S3)$ hold and let $\cU\in\cF$ be given by $(S1)$. For each $\mu>0$, let $\ovl m_\mu(\cdot)$ and $\ovl H$ be respectively the averaged metric and  the effective Hamiltonian corresponding to $\cU$. Then 
\begin{equation}\label{eq:Hbar-mbar}
	\ovl H(p) = \inf\{\mu>0 : \ovl m_\mu(y) \geq p\cdot y \hbox{ for all } y \in \R^n\}.
\end{equation}
Furthermore, for each $y \in \R^n$,
\begin{equation}\label{eq:mbar-Hbar} 
	\ovl{m}_\mu(y) = \sup\{ y \cdot q :  q\in\R^n \hbox{  such that } \ovl{H}(q) \le \mu\}
\end{equation}
and it is the solution of
\begin{equation}\label{eq:Hm}
	\ovl{H}(D\ovl{m}_\mu)=\mu \quad \text{in} \ \R^n \setminus \{0\}, 
	\;\;\; \hbox{ with } \;\;\; \ovl m_\mu(0)=0.
\end{equation}
\end{corollary}

Note that Corollary \ref{corr::avg-metric-H} can be seen as a homogenization result for the metric problem. Indeed, by rescaling 
$m_\mu^\eps(x,z,\omega) = \eps m_\mu \left( x/\eps,  z/\eps, \omega\right)$
we observe that $m_\mu^\eps$ satisfies, for $z_\eps=z/\eps\in U_\eps(\omega)$,
$$  \begin{cases}
       a\left(\frac x\eps,\omega\right) |Dm^\eps_\mu| = \mu & \hbox{ in }\   U_\eps(\omega) \setminus \{z_\eps\},\\
       m^\eps_\mu(\cdot,z_\eps,\omega)= 0& \hbox{ at } \{z_\eps\}
  \end{cases}
$$
and, as $\eps\to 0$, $m^\eps_\mu$ converges to $ \ovl m_\mu$, with $\ovl m_\mu$ a solution of problem \eqref{eq:Hm}. We prove in the next section, how this implies the homogenization of the original problem \eqref{eq:HJ_eps}. 

Finally, we give a characterization of the subgradient of the averaged metric $m_\mu$ corresponding to each ergodic constant $\mu = \ovl H(p)$. This result is used in the main proof of homogenization. 
\begin{corollary}\label{corr::p-unitvector}
Assume that $(A1),(A2),(A3)$, $(S1),(S2)$ and $(S3)$ hold. Let $\cU$ be given by $(S1)$, $\ovl H$ the corresponding effective Hamiltonian and $\ovl m_{\mu}$ be the averaged metric. Then, for each $p\in\R^n$ with $\ovl{H}(p)>0$, there exists $y\in\R^n$ with $|y|=1$ such that $p \in \partial \ovl m_{\ovl H(p)}(y)$.
\end{corollary}

\medskip\section{The Proof of Homogenization}\label{sec::ProofHomogenization}

It is well known (see for example \cite{AS:12}) that an intermediate step in the proof of homogenization of \eqref{eq:HJ_eps} is  the homogenization of the time-independent problem
   \begin{equation}\label{eq:HJ_stat}
    w^\eps + a\left (\frac{x}{\eps},\omega\right) |p+Dw^\eps|=0
    \; \hbox{ in } \;\;\;\; \R^n \times \Omega.
  \end{equation}

The next proposition summarizes the properties of the solution to \eqref{eq:HJ_stat}.

\begin{proposition}
Assume $(A2)$ and $(A3)$. For each $\eps>0$ and $p\in\R^n$, \eqref{eq:HJ_stat} has a unique solution $w^\eps = w^\eps(\cdot,\omega;p)$ such that, for all $\omega\in\Omega$, 
  \[||w^\eps (\cdot,\omega)||_\infty \leq \|a\|_{\infty} |p|.\]
Moreover, for  each compact set $K\subset  U_{i,\eps}(\omega)$, 
	\[\ \sup_{K} |Dw^\eps(\cdot,\omega;p)|\leq \left(\frac{\|a\|_{\infty}}{\inf_K |a(\cdot,\omega)|}+1\right)|p| \]
and, for all $p,q\in\R^n$
	\[\| w^\eps(\cdot,\omega;p) - w^\eps (\cdot,\omega;q) \|_{L^\infty}\leq \|a\|_{\infty} |p-q|. \]
\end{proposition}

We next establish a homogenization result for \eqref{eq:HJ_stat}, from where Theorem \ref{thm::homog-time} follows. For the proof of convergence we deal with the entire partition $\big(U_i(\omega)\big)_{i\in I_0}$ of $\R^n$. We recall that,  for each bounded connected component, the effective Hamiltonian $\ovl H_i \equiv 0$, while for unbounded components  $\ovl H_i>0$ if $i\in I^+$ and $\ovl H_i <0$ if $i\in I^-$.

\begin{proposition}\label{prop::homog-stat}
Assume $(A1),(A2),(A3)$, $(S1),(S2)$ and $(S3)$.   For each $i\in I$, let $\cU_i$ be given by $(S1)$, $\theta_i=\bP(\mathcal{U}_i)$ and $\ovl{H_i}$ the effective Hamiltonian corresponding to $\cU_i$. There exists $\ovl \delta_0>0$ and an event of full probability  $\widetilde \Omega\subseteq\Omega$ such that, for each $\omega\in\widetilde\Omega$ and $p\in\R^n$, the unique solution $w^\eps = w^\eps(\cdot,\omega;p)$ of \eqref{eq:HJ_stat} satisfies, for all $\delta\in(0,\ovl \delta_0)$, $R>0$ and $i\in I$,
  \begin{equation}\label{eq:w-unif-conv}
  \lim_{\eps \to 0} \sup_{x\in U_{i,\eps}^\delta(\omega)  \cap B_R}
  |w^\eps(x,\omega;p)+\ovl H_i(p)|=0
  \end{equation}
and, as $\eps\to0$,
  \begin{equation}\label{eq:w-weak-conv}
  w^\eps(\cdot,\omega;p) \stackrel{*}\rightharpoonup   \ovl w 
	: =\sum_{i\in I} \theta_i\; \ovl{H}_i(p) \hbox{ in }\ L^\infty(B_R).
  \end{equation}
\end{proposition}


We follow the approach developed in \cite{AS:13} and present a direct argument to show that Theorem \ref{thm::avg-metric} and Corollary \ref{corr::avg-metric-H} imply homogenization of \eqref{eq:HJ_stat}. We use the perturbed test function method introduced in \cite{Ev:89}, where we face the two usual difficulties, namely the blow up of the $m^\eps_\mu$'s near the boundaries of $ U_{i,\eps}(\omega)$ and their restriction to $U^\delta_{i,\eps}(\omega)$.

\begin{proof}[\bf Proof of Proposition \ref{prop::homog-stat}]
We divide the proof into three different steps. The first is about the sub-solution property, the second about the super-solution and the last about the weak convergence result. For the first we follow  the analogous proof of \cite{AS:13}. For the second step it is usually necessary to consider another class of metric problems corresponding to $H(-p,x,\omega)$. Here, however, this is not necessary, in view of the symmetry of the Hamiltonian.

To fix ideas, we consider, as usual, one connected component $U(\omega)$ of $\{x\in\R^n:a(x,\omega)>0\}$ and we drop the subscript $i$ in the following. Note that, in this case, $w^\eps\leq 0$.\smallskip

{\it Step 1.} We first show  that,  for all $\omega\in\tilde\Omega$, $p \in \R^n$, $\delta>0$ sufficiently small and $R>0$,
\begin{equation}\label{eq:sta-step1}
 	\lim_{\eps \to 0} \sup_{x\in U_\eps(\omega)\cap B_R} w^\eps(x,\omega) \leq -\ovl H(p).
\end{equation}
The proof uses a comparison argument similar to \cite{AS:12,AS:13}, but special care is needed to handle the restriction of $w^\eps$ to $U_\eps(\omega)$. Proposition \ref{prop::liminf-H}  is essential in establishing \eqref{eq:sta-step1}.

Since $w^\eps\leq 0$, the result is immediate if $\ovl{H}(p)=0$, hence we only  consider the case $\ovl{H}(p)>0$. We argue by contradiction. If \eqref{eq:sta-step1} is false, then there exists $\eta>0$ and $\delta>0$ such that, for every $\eps>0$ sufficiently small there exists $z_\eps\in U^\delta_\eps(\omega)\cap B_R$ so that 
$ w^\eps(z_\eps,\omega;p) > - \ovl H(p) + \eta. $
We may assume further, that as $\eps\to 0$, $z_\eps\to z_0$.   For $c>0$ a positive constant, to be conveniently chosen, define
	\[ v^\eps(x,\omega;p) = w^\eps(x,\omega; p) - w^\eps(z_\eps,\omega;p) 
 	 - c\eta \left(1+|x-z_\eps|^2\right)^{1/2}+c\eta.\]
It follows that, for $c=1/(4\|a\|_{\infty})$, $v^\eps(\cdot,\omega;p)$ satisfies, in the viscosity sense,
\begin{equation}\label{eq:supersol-metric}
  	a\left(\frac{x}{\eps},\omega\right)|p+Dv^\eps| \leq \ovl H(p) - \frac\eta2  
 	 \;\;\; \hbox{ in } \;\;\;\;\;\; D_{\eta,\eps}(\omega),
\end{equation}
where $ D_{\eta,\eps}(\omega) : = \left\{ x\in  U_\eps(\omega): v^\eps(x,\omega;p)\geq -\eta/4\right\}.$
Furthermore, there exists $r>0$ depending on $\eta, \|a\|_{\infty}$ and independent of $\eps$ such that 
$$ D_{\eta,\eps}(\omega)\subseteq B_r(z_\eps).$$
Let $\mu = \ovl H(p)>0$. Corollary \ref{corr::p-unitvector} implies that there exists a vector $|e|=1$ so that $p\in\partial \ovl m_\mu(e)$ and  
\begin{equation}\label{eq:ineq-p-unitvector}
 	0 = \ovl m_\mu (e) - p\cdot e \leq \ovl m_\mu(y)-p\cdot y\ \hbox{ for all } y\in\R^n. 
\end{equation}
In view of $(S3)$, there exists a sequence ${\eps_{j}}\to 0$, as $j\to \infty$, so that $\hat z:= z_0 - re\in  U_{\eps_{j}}^\delta(\omega)$ for all ${\eps_{j}}>0$.
Since $m_\mu(\cdot,\cdot,\omega)$ is the maximal solution of the metric problem \eqref{eq:metricpb-wp} for $p=0$, 
	\[m^{{\eps_{j}}}_\mu(x, \hat z,\omega;p) = 
	{\eps_{j}} m_\mu\left (\frac{x}{{\eps_{j}}},\frac{\hat z}{{\eps_{j}}},\omega\right) - p \cdot (x-\hat z)\]
satisfies
\begin{equation}\label{eq:subsol-metric}
  	a\left(\frac{x}{{\eps_{j}}},\omega\right)|p+Dm_\mu^{{\eps_{j}}}(\cdot,\hat z,\omega)|= \ovl H(p)  
 	 \hbox{ in } U_{\eps_j}(\omega)\setminus{ \{\hat z\}}.
\end{equation}
Recalling that $D_{\eta,{\eps_{j}}}(\omega)\subset {\eps_{j}} U(\omega)\setminus{ \{ \hat z\}}$,  by classical comparison arguments between viscosity sub- and super-solution, we have 
	\[\inf_{D_{\eta,{\eps_{j}}}(\omega)} 
  	\left( m^{{\eps_{j}}}_\mu(x,\hat z,\omega;p) - v^{{\eps_{j}}}(x,\omega;p)\right) =
  	\inf_{\partial D_{\eta,{\eps_{j}}}(\omega)} 
  	\left( m^{{\eps_{j}}}_\mu(x,\hat z,\omega;p) - v^{{\eps_{j}}}(x,\omega;p)\right). \]
Note that $z_{\eps_j}\in D_{\eta,{\eps_{j}}}(\omega)$ and thus
\begin{eqnarray*}
 	\inf_{D_{\eta,{\eps_{j}}}(\omega)} 
  	\left( m^{{\eps_{j}}}_\mu(x,\hat z,\omega;p) - v^{{\eps_{j}}}(x,\omega;p)\right) 
  	& \leq & m^{{\eps_{j}}}_\mu(z_{{\eps_{j}}},\hat z,\omega;p) - v^{{\eps_{j}}}(z_{{\eps_{j}}},\omega;p) \\
  	& \leq &  {\eps_{j}} m_\mu\left(\frac{z_{{\eps_{j}}}}{{\eps_{j}}},\frac{\hat z}{{\eps_{j}}},\omega\right)- 
	p\cdot(z_{{\eps_{j}}}-\hat z).
\end{eqnarray*}  
On the other hand $m_\mu (\cdot,\hat z,\omega) \equiv \infty $ on $\partial (U_{\eps_{j}}(\omega))$ and therefore the infimum cannot be achieved on $\partial (U_{\eps_{j}}(\omega))$. Thus, since $v^{{\eps_{j}}}(\cdot,\omega;p) = - \eta/4$ on $\partial D_{\eta,{{\eps_{j}}}}(\omega)\setminus \partial U_{\eps_{j}}(\omega)$, 
\[ 	\inf_{\partial D_{\eta,{\eps_{j}}}(\omega)}  
	\left(m^{{\eps_{j}}}_\mu(x,\hat z,\omega;p) - v^{{\eps_{j}}}(x,\omega;p)\right) 
 	=  \inf_{\partial D_{\eta,{\eps_{j}}}(\omega)} 
  	\left( {\eps_{j}} m_\mu\left(\frac{x}{{\eps_{j}}},\frac{\hat z}{{\eps_{j}}},\omega\right)
  	- p\cdot(x-\hat z) \right)+ \frac\eta4.
\]
The above and the fact that  $ D_{\eta,\eps}(\omega)\subseteq B_r(z_\eps)$ imply
\begin{eqnarray*}
 	\inf_{B_r(z_{{\eps_{j}}}) \cap  U_{\eps_j}(\omega)} \left(
     	{\eps_{j}} m_\mu\left(\frac{x}{{\eps_{j}}},\frac{\hat z}{{\eps_{j}}},\omega\right) 
     	- {\eps_{j}} m_\mu\left(\frac{z_{{\eps_{j}}}}{{\eps_{j}}},\frac{\hat z}{{\eps_{j}}},\omega\right)
     	- p\cdot(x-z_{{\eps_{j}}}) \right)  \leq - \frac\eta 4.
\end{eqnarray*}
Taking inferior limit as $j\to\infty$ and recalling that $z_{{\eps_{j}}}\to z_0$, we obtain, in light of Theorem \ref{thm::avg-metric} and Lemma \ref{lemma::up-to-bdry}, that
	\[\inf_{x\in B_r(z_0)}\big(\ovl m_\mu(x-\hat z) - \ovl m_\mu(z_0-\hat z) - p\cdot(x-z_0)\big) \leq - \frac\eta 4,\]
and thus
	\[\inf_{x\in B_r(z_0)} \big(\ovl m_\mu(x-z_0+re) -  \ovl m_\mu(re) - p\cdot(x-z_0)\big) =   r 
  	\inf_{y\in B_1(0)}    \big(\ovl m_\mu(e)- \ovl m_\mu(y)  - p\cdot y\big) \leq - \frac{\eta}{4}.\]
We arrived to a contradiction with \eqref{eq:ineq-p-unitvector}. \smallskip

{\it Step 2.} Arguing similarly, we establish the converse inequality 
\begin{equation*}
    	\lim_{\eps \to 0} \inf_{x\in U^\delta_\eps(\omega)\cap B_R} w^\eps(x,\omega;p)
    	\geq -\ovl H(p).
\end{equation*}

In view of the symmetry of $m_\mu(\cdot,\cdot,\omega)$, we now compare $w^\eps(\cdot/\eps,\cdot/\eps,\omega)$ with the downward cone $-m_\mu(\cdot/\eps,\cdot/\eps,\omega)$. Note, however, that symmetry plays a crucial role in this case and that, in general, one could not use $-m_\mu(\cdot/\eps,\cdot/\eps,\omega)$, but rather consider a different set of metric problems (see for example the case of level set convex Hamilton-Jacobi equations \cite{AS:13}).\smallskip

{\it Step 3.}
Finally, the weak convergence result \eqref{eq:w-weak-conv} is an immediate consequence of the two previous steps.  We present a sketch of the proof here for the sake of completeness. For details, we send the reader to \cite{CLS:09}. Recall that  $I=I^+\cup I^-\cup\{0\}$ with $\ovl H_0\equiv0$. Fix $\eta>0$, let $J$ be a finite subset of $I$ so that
	\[ \Big | B_R \setminus \bigcup_{j \in J}U_{j,\eps}(\omega) \Big | \leq \frac\eta2. \]
Let $\delta>0$ sufficiently small so that
	\[ \Big | B_R \setminus \bigcup_{j \in J}U_{j,\eps}^\delta(\omega) \Big| \leq \eta. \]
In light of \eqref{eq:w-unif-conv} and the ergodic theorem, we have that
\begin{equation}\label{eq:weakconvJ}
	\mathbf{1}_{\bigcup_{j\in J} (U_{j,\eps}^\delta(\omega))} w^\eps  =
	\sum_{j\in J} \mathbf{1}_{U_{j,\eps}^\delta(\omega)}w^\eps  \; \stackrel{*}{\rightharpoonup}\; 
	\sum_{j\in J} \theta_j \ovl{H}_j(p) \;\;\;\; \hbox{ in } \;\;\;\; L^\infty(\ B_R).
\end{equation}
Thus, for any test function $\phi \in L^\infty(B_R)$, it is easy to check that
	\[ \left| \int_{B_R} \phi \ovl w\,dx - \int_{B_R} \phi w^\eps \mathbf{1}_{\bigcup_{j\in J} (U_{j,\eps}^\delta(\omega))}\,dx \right| \leq \eta \|\phi\|_{\infty}. \]
Letting $\eta \to 0$, we deduce the desired result.
This concludes the proof of homogenization for the time-independent problem \eqref{eq:HJ_stat}.
\end{proof}

It is immediate to show that Theorem \ref{prop::homog-stat} implies Theorem \ref{thm::homog-time}. We only remind the ideas  below.
\begin{proof}[\bf Proof of Theorem \ref{thm::homog-time}]
It is enough to establish the uniform convergence of time-dependent solutions in each connected component $U_i(\omega)$. We use the classical perturbed test function method to show that, when $i\in I^+$
      \[\phi(x,t,\omega):= \limsup_{\substack{\eps\to 0\\ y\to x, s\to t \\ y\in U^\delta_{i,\eps}(\omega)} } u^\eps (y,t,\omega)\]
is a subsolution of the initial value problem \eqref{eq:HJs_avg} and, hence,  by the comparison principle deduce that $ \phi(x,t)\leq \bar u_i(x,t)$.
Similarly, one can prove the reverse inequality, where, in view of Lemma \ref{lemma::up-to-bdry}, the above liminf is taken on $U_i(\omega)$, unlike limsup above which is restricted to $U^\delta_i(\omega)$. The weak convergence for time-dependent solutions follows from \eqref{eq:w-weak-conv}. This completes the proof of homogenization for \eqref{eq:HJ_eps}.
\end{proof}

\bibliographystyle{acm} 

\begin{thebibliography}{10}

\bibitem{AK:81}
{\sc Akcoglu, M.~A., and Krengel, U.}
\newblock Ergodic theorems for superadditive processes.
\newblock {\em J. Reine Angew. Math. 323\/} (1981), 53--67.

\bibitem{AB:10}
{\sc Alvarez, O., and Bardi, M.}
\newblock Ergodicity, stabilization, and singular perturbations for
  {B}ellman-{I}saacs equations.
\newblock {\em Mem. Amer. Math. Soc. 204}, 960 (2010), vi+77.

\bibitem{AP:96}
{\sc Antal, P., and Pisztora, A.}
\newblock On the chemical distance for supercritical {B}ernoulli percolation.
\newblock {\em Ann. Probab. 24}, 2 (1996), 1036--1048.

\bibitem{AL:98}
{\sc Arisawa, M., and Lions, P.-L.}
\newblock On ergodic stochastic control.
\newblock {\em Comm. Partial Differential Equations 23}, 11-12 (1998),
  2187--2217.

\bibitem{AS:12}
{\sc Armstrong, S.~N., and Souganidis, P.~E.}
\newblock Stochastic homogenization of {H}amilton-{J}acobi and degenerate
  {B}ellman equations in unbounded environments.
\newblock {\em J. Math. Pures Appl. (9) 97}, 5 (2012), 460--504.

\bibitem{AS:13}
{\sc Armstrong, S.~N., and Souganidis, P.~E.}
\newblock Stochastic homogenization of level-set convex {H}amilton-{J}acobi
  equations.
\newblock {\em Int. Math. Res. Not. 15}, 2 (2013), 3420â3449.

\bibitem{AT:13}
{\sc Armstrong, S.~N., and Tran., H.~V.}
\newblock Stochastic homogenization of viscous {H}amilton-{J}acobi equations
  and applications.
\newblock {\em preprint, arXiv:1310.1749 [math.AP]\/} (2013).

\bibitem{ATY:13}
{\sc Armstrong, S.~N., Tran., H.~V., and Yu, Y.}
\newblock Stochastic homogenization of a nonconvex {H}amilton-{J}acobi
  equation.
\newblock {\em preprint, arXiv:1311.2029 [math.AP]\/} (2013).

\bibitem{BCD:97}
{\sc Bardi, M., and Capuzzo-Dolcetta, I.}
\newblock {\em Optimal control and viscosity solutions of
  {H}amilton-{J}acobi-{B}ellman equations}.
\newblock Systems \& Control: Foundations \& Applications. Birkh\"auser Boston,
  Inc., Boston, MA, 1997.
\newblock With appendices by Maurizio Falcone and Pierpaolo Soravia.

\bibitem{B:93}
{\sc Barles, G.}
\newblock Discontinuous viscosity solutions of first-order {H}amilton-{J}acobi
  equations: a guided visit.
\newblock {\em Nonlinear Anal. 20}, 9 (1993), 1123--1134.

\bibitem{B:94}
{\sc Barles, G.}
\newblock {\em Solutions de viscosit\'e des \'equations de
  {H}amilton-{J}acobi}, vol.~17 of {\em Math\'ematiques \& Applications
  (Berlin)}.
\newblock Springer-Verlag, Paris, 1994.

\bibitem{B:07}
{\sc Barles, G.}
\newblock Some homogenization results for non-coercive {H}amilton-{J}acobi
  equations.
\newblock {\em Calc. Var. Partial Differential Equations 30}, 4 (2007),
  449--466.

\bibitem{BS:98}
{\sc Barles, G., and Souganidis, P.~E.}
\newblock A new approach to front propagation problems: theory and
  applications.
\newblock {\em Arch. Rational Mech. Anal. 141}, 3 (1998), 237--296.

\bibitem{BJ:90}
{\sc Barron, E.~N., and Jensen, R.}
\newblock Semicontinuous viscosity solutions for {H}amilton-{J}acobi equations
  with convex {H}amiltonians.
\newblock {\em Comm. Partial Differential Equations 15}, 12 (1990), 1713--1742.

\bibitem{CS:13}
{\sc Cardaliaguet, and Souganidis, P.~E.}
\newblock Homogenization and enhancement of the $g$-equation in random
  environments.
\newblock {\em Comm. Pure Appl. Math.\/} (to appear).

\bibitem{C:10}
{\sc Cardaliaguet, P.}
\newblock Ergodicity of {H}amilton-{J}acobi equations with a noncoercive
  nonconvex {H}amiltonian in {$\mathbb R^2/\mathbb Z^2$}.
\newblock {\em Ann. Inst. H. Poincar\'e Anal. Non Lin\'eaire 27}, 3 (2010),
  837--856.

\bibitem{CLS:09}
{\sc Cardaliaguet, P., Lions, P.-L., and Souganidis, P.~E.}
\newblock A discussion about the homogenization of moving interfaces.
\newblock {\em J. Math. Pures Appl. (9) 91}, 4 (2009), 339--363.

\bibitem{CB:03}
{\sc Craciun, B., and Bhattacharya, K.}
\newblock Homogenization of a {H}amilton-{J}acobi equation associated with the
  geometric motion of an interface.
\newblock {\em Proc. Roy. Soc. Edinburgh Sect. A 133}, 4 (2003), 773--805.

\bibitem{CIL:92}
{\sc Crandall, M.~G., Ishii, H., and Lions, P.-L.}
\newblock User's guide to viscosity solutions of second order partial
  differential equations.
\newblock {\em Bull. Amer. Math. Soc. (N.S.) 27}, 1 (1992), 1--67.

\bibitem{Ev:89}
{\sc Evans, L.~C.}
\newblock The perturbed test function method for viscosity solutions of
  nonlinear {PDE}.
\newblock {\em Proc. Roy. Soc. Edinburgh Sect. A 111}, 3-4 (1989), 359--375.

\bibitem{Ev:92}
{\sc Evans, L.~C.}
\newblock Periodic homogenisation of certain fully nonlinear partial
  differential equations.
\newblock {\em Proc. Roy. Soc. Edinburgh Sect. A 120}, 3-4 (1992), 245--265.

\bibitem{GM:04}
{\sc Garet, O., and Marchand, R.}
\newblock Asymptotic shape for the chemical distance and first-passage
  percolation on the infinite bernoulli cluster.
\newblock {\em ESAIM: Probability and Statistics 8\/} (9 2004), 169--199.

\bibitem{Gr:99}
{\sc Grimmett, G.}
\newblock {\em Percolation}, second~ed., vol.~321 of {\em Grundlehren der
  Mathematischen Wissenschaften [Fundamental Principles of Mathematical
  Sciences]}.
\newblock Springer-Verlag, Berlin, 1999.

\bibitem{IM:08}
{\sc Imbert, C., and Monneau, R.}
\newblock Homogenization of first-order equations with
  {$(u/\epsilon)$}-periodic {H}amiltonians. {I}. {L}ocal equations.
\newblock {\em Arch. Ration. Mech. Anal. 187}, 1 (2008), 49--89.

\bibitem{I:87}
{\sc Ishii, H.}
\newblock Perron's method for {H}amilton-{J}acobi equations.
\newblock {\em Duke Math. J. 55}, 2 (1987), 369--384.

\bibitem{I:00}
{\sc Ishii, H.}
\newblock Almost periodic homogenization of {H}amilton-{J}acobi equations.
\newblock In {\em International {C}onference on {D}ifferential {E}quations,
  {V}ol. 1, 2 ({B}erlin, 1999)}. World Sci. Publ., River Edge, NJ, 2000,
  pp.~600--605.

\bibitem{KRV:06}
{\sc Kosygina, E., Rezakhanlou, F., and Varadhan, S. R.~S.}
\newblock Stochastic homogenization of {H}amilton-{J}acobi-{B}ellman equations.
\newblock {\em Comm. Pure Appl. Math. 59}, 10 (2006), 1489--1521.

\bibitem{KV:08}
{\sc Kosygina, E., and Varadhan, S. R.~S.}
\newblock Homogenization of {H}amilton-{J}acobi-{B}ellman equations with
  respect to time-space shifts in a stationary ergodic medium.
\newblock {\em Comm. Pure Appl. Math. 61}, 6 (2008), 816--847.

\bibitem{PLL:82}
{\sc Lions, P.-L.}
\newblock {\em Generalized solutions of {H}amilton-{J}acobi equations}, vol.~69
  of {\em Research Notes in Mathematics}.
\newblock Pitman (Advanced Publishing Program), Boston, Mass.-London, 1982.

\bibitem{LPV}
{\sc Lions, P.-L., Papanicolaou, G., and Varadhan, S.}
\newblock Homogenization of {H}amilton-{J}acobi equations.
\newblock Preprint, 1987.

\bibitem{LS:05}
{\sc Lions, P.-L., and Souganidis, P.~E.}
\newblock Homogenization of ``viscous'' {H}amilton-{J}acobi equations in
  stationary ergodic media.
\newblock {\em Comm. Partial Differential Equations 30}, 1-3 (2005), 335--375.

\bibitem{LS:10}
{\sc Lions, P.-L., and Souganidis, P.~E.}
\newblock Stochastic homogenization of {H}amilton-{J}acobi and
  ``viscous''-{H}amilton-{J}acobi equations with convex
  nonlinearities---revisited.
\newblock {\em Commun. Math. Sci. 8}, 2 (2010), 627--637.

\bibitem{OS:88}
{\sc Osher, S., and Sethian, J.~A.}
\newblock Fronts propagating with curvature-dependent speed: algorithms based
  on {H}amilton-{J}acobi formulations.
\newblock {\em J. Comput. Phys. 79}, 1 (1988), 12--49.

\bibitem{RT:00}
{\sc Rezakhanlou, F., and Tarver, J.~E.}
\newblock Homogenization for stochastic {H}amilton-{J}acobi equations.
\newblock {\em Arch. Ration. Mech. Anal. 151}, 4 (2000), 277--309.

\bibitem{S:09}
{\sc Schwab, R.~W.}
\newblock Stochastic homogenization of {H}amilton-{J}acobi equations in
  stationary ergodic spatio-temporal media.
\newblock {\em Indiana Univ. Math. J. 58}, 2 (2009), 537--581.

\bibitem{S:95}
{\sc Souganidis, P.~E.}
\newblock Front propagation: theory and applications.
\newblock In {\em Viscosity solutions and applications ({M}ontecatini {T}erme,
  1995)}, vol.~1660 of {\em Lecture Notes in Math.} Springer, Berlin, 1997,
  pp.~186--242.

\bibitem{S:99}
{\sc Souganidis, P.~E.}
\newblock Stochastic homogenization of {H}amilton-{J}acobi equations and some
  applications.
\newblock {\em Asymptot. Anal. 20}, 1 (1999), 1--11.

\end{thebibliography}

\end{document}